\documentclass[onefignum,onetabnum]{siamart171218}

\usepackage{graphics}
\usepackage{graphicx}
\usepackage{amsmath,amsfonts,amssymb}
\ifpdf
  \DeclareGraphicsExtensions{.eps,.pdf,.png,.jpg}
\else
  \DeclareGraphicsExtensions{.eps}
\fi

\newcommand{\TheTitle}
{Variational Properties of Matrix Functions via the 
Generalized Matrix-Fractional Function}
\newcommand{\TheAuthors}{J. V. Burke, Y. Gao, and T. Hoheisel}

\headers{The Generalized Matrix-Fractional Function}
{\TheAuthors}

\title{{\TheTitle}}

\author{
	James V. Burke\thanks{Department of Mathematics, University of Washington, Seattle, WA 98195 (\email{jvburke@uw.edu}). Research is supported in part by the National Science Foundation under grant number DMS­-1514559.} 
	\and
	Yuan Gao\thanks{Department of Applied Mathematics, University of Washington, Seattle, WA 98195 (\email{yuangao@uw.edu}).}
	\and
	Tim Hoheisel\thanks{McGill University, 805 Sherbrooke St West, Room1114, Montr\'eal, Qu\'ebec, Canada H3A 0B9 \email{tim.hoheisel@mcgill.ca})}
}

\usepackage{amsopn}

\newtheorem{example}[theorem]{Example}

\def\del{\delta}
\newcommand{\N}{\mathbb{N}}

\newcommand{\goesto}{\rightarrow}

\newcommand{\aff}{\mathrm{aff}\,}
\def\argmin{ \mathop{{\rm argmin}}}
\def\argmax{ \mathop{{\rm argmax}}}

\newcommand{\rbd}{\mathrm{rbd}\,}
\newcommand{\co}{\mathrm{conv}\,}
\newcommand{\cl}{\mathrm{cl}\,}
\newcommand{\clco}{\mathrm{\overline{conv}}\,}

\newcommand{\dom}{\mathrm{dom}\,}
\newcommand{\epi}{\mathrm{epi}\,}

\newcommand{\inter}{\mathrm{int}\,}
\newcommand{\pos}{\mathrm{pos}\,}
\newcommand{\ri}{\mathrm{ri}\,}
\newcommand{\barr}{\mathrm{bar}\,}
\newcommand{\rge}{\mathrm{rge}\,}
\newcommand{\Ker}{\mathrm{Ker}}
\newcommand{\Rge}{\mathrm{Rge}}
\newcommand{\hzn}{\mathrm{hzn}\,}
\newcommand{\lin}{\mathrm{span}\,}

\newcommand{\p}{\partial}

\newcommand{\R}{\mathbb{R}}

\newcommand{\Rnm}{\R^{n\times m}}

\newcommand{\bS}{\mathbb{S}}
\newcommand{\Sn}{\bS^n}

\newcommand{\rank}{\mathrm{rank}\,}
\newcommand{\rbar}{\overline{\mathbb R}}
\newcommand{\Rbar}{\overline{\mathbb R}}
\newcommand{\eR}{\mathbb R\cup\{+\infty\}}
\newcommand{\bR}{\mathbb{R}}

\def\Rnm{\bR^{n\times m}}
\def\bS{\mathbb{S}}
\def\Sn{\bS^n}
\def\Snp{\bS^n_+}

\newcommand{\tr}{\mathrm{tr}\,}

\newcommand{\cone}{\mathrm{cone}\;}

\newcommand{\bB}{\mathbb{B}}
\newcommand{\bN}{\mathbb{N}}
\newcommand{\bE}{\mathbb{E }}

\newcommand{\map}[3]{#1 :#2\rightarrow #3}
\newcommand{\ip}[2]{\left\langle #1,\, #2\right\rangle}
\newcommand{\half}{\frac{1}{2}}
\newcommand{\norm}[1]{\left\Vert #1\right\Vert}
\newcommand{\tnorm}[1]{\left\Vert #1\right\Vert_2}

\newcommand{\set}[2]{\left\{#1\,\left\vert\; #2\right.\right\}}
\newcommand{\ncone}[2]{N_{#2}\left(#1\right)}

\newcommand{\support}[2]{\sig_{#2}\left(#1\right)}

\newcommand{\indicator}[2]{\del_{#2}\left(#1\right)}

\newcommand{\gauge}[2]{\gam_{#2}\left(#1\right)}
\newcommand{\barrier}{{\mathrm{bar}\,}}

\def\sd{\partial}

\newcommand{\bx}{{\bar{x}}}
\newcommand{\by}{{\bar{y}}}
\newcommand{\bz}{{\bar{z}}}

\newcommand{\bt}{{\bar{t}}}
\newcommand{\barX}{{\bar X}}
\newcommand{\barY}{{\bar Y}}

\newcommand{\barV}{{\bar V}}
\newcommand{\barW}{{\bar W}}

\newcommand{\gam}{\gamma}

\newcommand{\sig}{\sigma}

\newcommand{\lam}{\lambda}

\def\eps{\epsilon}

\newcommand{\cD}{\mathcal{D}}
\newcommand{\cE}{\mathcal{E}}
\newcommand{\cF}{\mathcal{F}}
\newcommand{\cK}{\mathcal{K}}

\newcommand{\cS}{\mathcal{S}}
\newcommand{\cT}{\mathcal{T}}
\newcommand{\cU}{\mathcal{U}}

\newcommand{\cV}{\mathcal{V}}

\newcommand{\kS}{\mathfrak{S}}

\newcommand{\AND}{\ \mbox{ and }\ }

\newcommand{\infconv}{\mathbin{\mbox{\small$\square$}}}

\newcommand{\hcVp}{{\cV^{1/2}_A}}

\newcommand{\vphi}{\varphi}

\ifpdf
\hypersetup{
  pdftitle={\TheTitle},
  pdfauthor={\TheAuthors}
}
\fi

\begin{document}

\maketitle
\begin{abstract}
We show that many important convex matrix functions
can be represented as the partial infimal projection of the generalized
matrix fractional  (GMF) and a relatively simple convex function.
This representation provides 
conditions under which such functions are closed and proper as well
as  formulas for the ready computation of both 
their conjugates and subdifferentials.
Particular instances yield
all weighted Ky Fan norms and squared gauges on $\R^{n\times m}$,
and as an example we show that
all variational Gram functions are representable as squares of gauges.
Other instances yield weighted
sums of the Frobenius and nuclear norms. 
The scope of applications is large and the range
of variational properties and insight is fascinating and fundamental.
An important byproduct of these representations is that they 
lay the foundation for a smoothing approach to many matrix functions 
on the interior
of the domain of the GMF  function.
\end{abstract}

%



\maketitle

\pagestyle{myheadings}
\thispagestyle{plain}

\section{Introduction}\label{sec:Intro}

\noindent
The {\em generalized matrix-fractional (GMF)} function was introduced
by Burke and Hoheisel in \cite{BuH15} where it is shown to unify a number of tools
and concepts for matrix optimization including optimal value functions
in quadratic programming, nuclear norm optimization, multi-task learning,
and, of course, the matrix fractional function. In the present paper
we expand the number of applications to include all   {\em Ky Fan norms}, 
matrix {\em gauge functionals}, 
and {\em variational Gram functions} introduced by Jalali, Fazel and Xiao
in \cite{JFX17}. Our analysis includes descriptions
of the variational properties of these functions such as formulas
for their convex conjugates and their subdifferentials. 

Set  $\bE:=\Rnm\times \Sn$, where $\Rnm$ and $\Sn$ are the 
linear spaces of real
$n\times m$ matrices and (real) symmetric $n\times n$ matrices,
respectively. 
Given $(A,B)\in \R^{\ell\times n}\times\R^{\ell\times m}$ with $\rge B \subset \rge A$,
recall that the GMF function $\varphi$
is defined as the support function of
the graph of the matrix valued mapping 
$Y\mapsto -\half YY^T$ over the manifold $\set{Y\in\Rnm}{AY=B}$, i.e.,  
$\map{\varphi}{\bE}{\eR}$ is given by
\begin{equation}\label{eq:GMF support}
\varphi(X,V):=\sup\set{\ip{(Y,W)}{(X,V)}}{(Y,W)\in\cD(A,B)},
\end{equation}
where 
\begin{equation}\label{eq:D}
\cD(A,B) :=\set{\left(Y,-\half YY^T\right)\in \bE}{Y\in \mathbb R^{n\times m}:\;AY=B}.
\end{equation}
A closed form expression for $\varphi$ is derived in \cite[Theorem 4.1]{BuH15}
where it is also shown that $\varphi$ is smooth on
the (nonempty) interior of its domain. 

Our study focuses on  functions $\map{p}{\Rnm}{\Rbar=\R\cup\{\pm\infty\}}$ 
representable
as the partial infimal projection 
 \begin{equation}\label{eq:p1}
p(X):=\inf_{V\in\bS^n} \vphi(X,V)+h(V),
\end{equation}
where $\map{h}{\Sn}{\eR}$ is closed, proper, convex.
Different functions $h$ illuminate different variational properties
of the matrix $X$. 
For example, when $h:=\ip{U}{\cdot}$ for $U\in\Sn_{++}$ and both
$A$ and  $B$ are  zero, then $p$ is a weighted nuclear norm
where the weights depend on any ``square root'' of $U$
(see Corollary \ref{cor:h linear}). Among the 
consequences of the representation \eqref{eq:p1} 
are conditions under which $p$ is closed and proper as well
as formulas for the ready computation of both the conjugate $p^*$ and
the subdifferential $\sd p$
(Section \ref{sec:Main}). As an application of our general results,
we give more detailed explorations in the cases where $h$ is 
a support function (Section \ref{sec:Support Func}) 
or an indicator function (Section \ref{sec:Indicator}). We illustrate
these results with specific instances. For example,
we obtain all weighted squared gauges on $\R^{n\times m}$, cf.~Corollary \ref{cor:Gauge case II}, as well as  
a complete characterization of variational  Gram functions  \cite{JFX17} and their conjugates.
In addition, we show that all variational Gram functions are representable as squares of gauges, cf.~Proposition \ref{prop:VGF conjugate}.
Other choices yield weighted
sums of Frobenius and nuclear norms  \cite[Corollary 5.9]{BuH15}. 
The scope of applications is large and the range
of variational properties  is fascinating and fundamental.

Beyond 
the variational results of this paper, 
there is a compelling but unexplored
computational aspect:
Hsieh and Olsen \cite{HsO 14} show that  
\eqref{eq:p1} with $h=\half\tr (\cdot)$ yields a
smoothing approach to optimization problems involving the nuclear norm.
More generally, observe that many
matrix optimization problems  take the form
\begin{equation}\label{eq:P}
\min_{X\in\bR^{n\times m}} f(X)+p(X)\tag{$P$},
\end{equation}
where $f, p:\bR^{n\times m}\to \bR\cup\{+\infty\}$.
The function $f$ is thought of as the primary objective
and is often smooth or convex  while $p$ is typically
a structure inducing convex function. 
Using the representation \eqref{eq:p1}, the problem \eqref{eq:P}  can be written as
\begin{equation}\label{eq:embedd optimization}
\min_{(X,V)\in\bE} f(X)+\vphi(X,V)+h(V).
\end{equation}
This reformulation allows one to exploit the smoothness of $\varphi$ on the interior of its domain. For example, if both $f$ and $h$ are smooth, one can employ
a damped Newton, or path following approach to solving \eqref{eq:P}.
We emphasize, that this is not the goal or intent of this paper, however,
our results provide the basis for future investigations along a variety of such numerical and theoretical avenues.
 
The paper is organized as follows: In  Section \ref{sec:Prelim} we provide  the tools from convex  analysis and some basic properties of the GMF function.  Section \ref{sec:Main} contains the general theory for partial  infimal projections of the form \eqref{eq:p1}.  In Section \ref{sec:Support Func} we specify $h$ in \eqref{eq:p1} to be a support function of some closed, convex set $\cV\subset \bS^n$. In  Section \ref{sec:Indicator} we choose $h$ to be the indicator of such set. In particular,  this yields  powerful results  on  variational Gram functions  and Ky Fan norms in 
Sections \ref{sec:VGF} and \ref{sec:VGF and Ky Fan}.  We close out with some final remarks in Section \ref{sec:Final} and  supplementary material in Section \ref{sec:Appendix}.
\smallskip

\noindent
{\em Notation:} For a linear transformation $L$ between finite dimensional linear spaces, we write  
$\rge L$ and  $\ker L$ for its {\em range} and  {\em kernel}, respectively.
For a given choice of bases, every such linear 
transformation has a matrix representation
for some $A\in \R^{\ell\times n}$. Therefore, we also
write $\rge A$ and  $\ker A$ for the {\em range} and {\em kernel}, respectively, 
considering $A$ as a linear map between
$\R^n$ and $\R^\ell$.
Again, for $A\in \R^{\ell\times n}$, we set 
$$
\begin{aligned}\Ker_rA&:=\set{X\in \R^{n\times r}}{AX=0}
=\set{X\in \R^{n\times r}}{\rge X\subset \ker A},
\\
\Rge_rA
&:= \set{Y\in \R^{\ell\times r}}{\exists\, X\in \R^{n\times r}\, :\ Y=AX}
=\set{Y\in \R^{\ell\times r}}{\rge Y\subset \rge A}
\end{aligned}
$$ 
and write
$\Ker A$ or $\Rge A$ when the choice of $r$ is clear.
Observe that $\Ker_1 A=\ker A$, $\Rge_1A=\rge A$, and
$(\Ker_rA)^\perp=\Rge_rA^T$. We equip any matrix space with the (Frobenius) inner product $\ip{X}{Y}:=\tr(X^TY)$.
The {\em Moore-Penrose pseudoinverse}  \cite{HJ85} of $A$ is denoted by $A^\dagger$. The set of all  $n\times n$ symmetric matrices   is given by $\bS^n$.  
The positive and negative semidefinite cone are denoted by $\bS_+^n$ and $\bS^n_-$, respectively. 

For two sets $S,T$ in the same real linear space  their {\em Minkowski sum}
is
$
S+T:=\set{s+t}{s\in S,\; t\in T}.
$ 
For $I\subset  \R$ we also put 
$
I\cdot S:=\set{\lambda s}{\lambda\in I,\; s\in S}.
$

\section{Preliminaries}\label{sec:Prelim}

\subsection*{Tools from convex analysis}

Let $(\cE,\ip{\cdot}{\cdot})$ be a finite-dimensional Euclidean space with  induced norm   
$\|\cdot\|:=\sqrt{\ip{\cdot}{\cdot}}$.  The closed  $\eps$-ball about a point $x\in\cE$ is denoted by $B_\eps(x)$.  
%
%
%
Let   $S\subset\cE$ be nonempty.   
The (topological) {\em closure} and {\em interior} of  $S$ are denoted by $\cl S$ and $\inter S$,  respectively.  The {\em (linear) span} of $S$ is denoted by $\lin S$. 
The {\em affine hull} of $S$, denoted $\aff S$, is the intersection of all affine sets 
containing $S$, while
the {\em convex hull} of $S$, denoted $\co S$, is the intersection of all convex sets containing $S$.
Its closure (the {\em closed convex hull}) is $\clco S:=\cl(\co S)$.   
The {\em conical} and {\em convex conical} hull of $S$ are given by 
\(
\pos S:=\set{\lambda x}{x\in S,\;\lambda\geq 0},\) 
and 
\(
\cone S:=\set{\sum_{i=1}^r\lambda_i x_i}{r\in \bN,\; x_i\in S,\;\lambda_i\geq 0},
\)
respectively,
with $\cone S=\pos(\co S)=\co(\pos S)$. 
The closure of the latter is 
$
\overline \cone S:=\cl (\cone S).
$

The {\em relative interior} of a convex set $C\subset \cE$, denoted $\ri C$, is the interior
of $C$ relative to its affine hull. By \cite[Section 6.2]{BaC11}, we have
\begin{equation}\label{eq:RintChar}
x\in \ri C\iff\pos(C-x)=\lin (C-x).
\end{equation}
\noindent
The {\em polar set} of $S\subset \cE$ is defined by 
\(
S^\circ:=\set{v\in \cE}{\ip{v}{x}\leq 1\;(x\in S)},
\)
and the {\em horizon cone} is the closed cone
\(
S^\infty:=\set{v\in \cE}{\exists \{\lambda_k\}\downarrow 0,\; \{x_k\in S\}:\; \lambda_k x_k\to v}.
\)
For a convex set $C\subset \cE$, $C^\infty$  coincides with the {\em recession cone} of the closure of $C$, i.e.
\begin{equation}\label{eq:horizon recession}
C^\infty=\set{v}{x+tv\in \cl C\;(t\geq 0, \;x\in C)}
=\set{y}{C+y\subset C}.
\end{equation}



\noindent
For $\map{f}{\cE}{\rbar}$ its {\em domain} and {\em epigraph} are given by
\(
\dom f:=\set{x\in\cE}{ f(x)<+\infty}\)
and \(\epi f:=\set{(x,\alpha)\in\cE\times \R}{f(x)\leq \alpha},
\)
respectively.
We say $f$ is {\em proper} if $f(x)>-\infty$ for all $x\in\dom f\ne\emptyset$.
We call $f$   {\em convex} if its epigraph $\epi f$ is convex,  and {\em closed} (or {\em lower semicontinuous}) if $\epi f$ is closed. If $f$ is proper, we call it {\em positively homogeneous}  if $\epi f$ is a cone, and {\em sublinear} if $\epi f$ is a convex cone. 
In what follows we use the following abbreviations:
\[
\Gamma(\cE)  :=  \set{f:\cE\to \eR}{f\;\text{proper, convex}},\;\Gamma_0(\cE) :=  \set{f\in \Gamma(\cE)}{f\;\text{closed}}.
\]
The {\em lower semicontinuous hull} $\cl f$ and the  {\em horizon function} $f^\infty$ of $f$ are defined through the relations
\(
\cl(\epi f)=\epi \cl f \AND\epi f^\infty = (\epi f)^\infty,
\)  
respectively. For $f\in \Gamma_0(\cE)$, $f^\infty$ is also known as the 
 {\em recession function} \cite[p.~66]{RTR70} or the   
{\em asymptotic function}  \cite{AuT03,HUL 01}.
The {\em horizon cone of a function} $f$ is defined as 
\(
\hzn f:=\set{x}{f^\infty(x)\leq 0},
\)
and for $f\in \Gamma_0$, we have 
\(
\hzn f=\set{x}{f(x)\le \mu}^\infty
\)
for  $\mu\in \R$ such that $\set{x}{f(x)\le \mu}\ne\emptyset$ \cite[Theorem 8.7]{RTR70}.

For a convex function $\map{f}{\cE}{\eR}$ its {\em subdifferential} at  
$\bar x\in \dom f$ is given by
\(
\partial f(\bar x):=\set{v\in \cE}{f(x)\geq f(\bar x)+\ip{v}{x-\bar x}\, (x\in\cE)}.
\)
For  $f\in\Gamma_0(\cE)$, we have
\(
\ri(\dom f)\subset \dom \p f\subset \dom f,
\)
see e.g. \cite[p.~227]{RTR70}, where
$\dom \partial f:=\set{x\in \cE}{\p f(x)\neq \emptyset}$ is the {\em domain of the subdifferential}.
 
For a function $f:\cE\to\rbar$ its {\em (Fenchel) conjugate}   $f^*:\cE\to \rbar$ is given by
\(
f^*(y):=\sup_{x\in\cE}\{\ip{x}{y}-f(x)\},
\)
and $f\in \Gamma_0(\cE)$ if and only if $f=f^{**}:=(f^*)^*$ is proper.

\noindent
Given a nonempty $S\subset \cE$, its {\em indicator function} $\delta_S:\cE\to\eR$ is given by $\indicator{x}{S} =0$ for $x\in S$ and
$+\infty$ otherwise.
The indicator of $S$ is convex if and only if  $S$ is a convex set, in which case the {\em normal cone} of $S$ at $\bar x\in S$ is given by
\(
\ncone{\bar x}{S}:=\partial \delta_S(\bar x)=\set{v\in \cE}{\ip{v}{x-\bar x}\leq 0 \;(x\in S)}.
\)
The  {\em support function} $\sigma_S:\cE\to \eR$ and the {\em gauge function} $\gamma_S:\cE\to\eR$ of a nonempty set $S\subset \cE$ are given   respectively by 
\(
 \support{x}{S}:=\sup_{v\in S}\ip{v}{x}\AND \gauge{x}{S}:=\inf\set{t \geq  0}{x\in tS}.
\)
Here we use the standard convention that  $\inf \emptyset=+\infty$. 

\noindent
Given $C\subset \cE$ is closed and convex, the {\em barrier cone} of $C$
is defined by 
$
\barrier C:=\dom \sigma_{C}.
$
The closure of the barrier cone of $C$ and the horizon cone are paired in polarity, i.e. 
\begin{equation}\label{eq:BarrierPolarity}
(\barrier C)^\circ=C^\infty\AND \cl(\barrier C)=(C^\infty)^\circ.
\end{equation}
\noindent
  For two functions $f_1,f_2:\cE\to \rbar$, their  {\em infimal convolution} 
is
\[
(f_1\infconv f_2)(x):=\inf_{y\in \cE}\{f_1(x-y)+f_2(y)\}\quad (x\in\cE).
\]


\subsection*{The generalized matrix-fractional function} 
As noted in the introduction, the GMF function is the support function
of $\cD (A,B)$ given in  \eqref{eq:D}. Hence, we write  
\begin{equation}\label{eq:GMF}
 \varphi(X,V)=\sigma_{ \cD (A,B)} (X,V)
\end{equation}
and also  refer to $\sigma_{ \cD (A,B)}$ as the GMF function.
From \cite[Theorem 4.1]{BuH15},  we obtain the formula
\begin{equation}\label{eq:DefPhi}
\varphi(X,V)=
\left\{\begin{array}{lcl}\frac{1}{2} \tr\!\left( \binom{X}{B}^TM(V)^\dagger \binom{X}{B}  \right) 
& {\rm if}  & \rge \binom{X}{B}\subset \rge M(V), \; V\in\cK_A ,\\
+\infty & {\rm else},
\end{array}\right. 
\end{equation}
where 
$(A,B)\in \R^{\ell\times n}\times\R^{\ell\times m}$ with $\rge B \subset \rge A$ and $\cK_A$ is the cone of all symmetric matrices that are positive semidefinite with respect to the subspace $\ker A$, i.e.
 \begin{equation}\label{eq:K}
\cK_A:=\set{V\in\Sn}{u^TVu \ge 0\;(u \in \ker A)},
\end{equation}
and $M(V)^\dagger$ is the Moore-Penrose pseudoinverse of the {\em bordered matrix}
\begin{equation}\label{eq:M(V)}
M(V)=\begin{pmatrix}V&A^T\\ A&0\end{pmatrix}.
\end{equation}
The {\em matrix-fractional function} \cite{BoV 04,Dat 05}
is obtained by setting $A$ and $B$ to zero.

The GMF function $\vphi=\sigma_{\cD(A,B)}$ appears in Burke and Hoheisel \cite{BuH15} and Burke, Hoheisel and Gao \cite{BGH17}, where it is shown that
\begin{equation}\label{eq:dom support D}
\begin{aligned}
\dom \vphi=\dom \partial \vphi
 &=
\set{(X,V)\in \bE}{  \rge \binom{X}{B}\subset \rge M(V), \; 
V\in\cK_A   },
\\
 \inter (\dom \vphi)&=
\set{(X,V)\in \bE}{  \rge \binom{X}{B}\subset \rge M(V), \; 
V\in\inter \cK_A   }\neq \emptyset.
\end{aligned}
\end{equation}
For a deeper understanding of the support function $\vphi$, 
a description of the closed convex hull of the (nonconvex) set $\cD(A,B)$ is critical.  An arduous representation of $\clco \cD(A,B)$ 
was obtained in \cite[Proposition 4.3]{BuH15}. 
A much simpler and more versatile expression 
was proven in \cite[Theorem 2]{BGH17}, see below.  
The key ingredient in the newer expression is the (closed, convex) cone $\cK_A$ defined in \eqref{eq:K}, which  reduces to $\bS_+^n$ when $ A=0$. 
We briefly summarize the geometric and topological properties of $\cK_A$ useful to our study. These follow from \cite[Proposition 1]{BGH17} (by setting   $\cS=\ker A$).

\begin{proposition}\label{prop:K_A} For $A\in \R^{\ell\times n}$ let $P\in \R^{n\times n}$ be the orthogonal projection onto $\ker A$ and let  $\cK_A$ be given by \eqref{eq:K}. Then the following hold:
\begin{itemize}
\item[(a)]  $\cK_A=\set{V\in \bS^n}{PVP\succeq 0}$.
\item[(b)]
$
\cK_A^{\circ}=\cone \set{-vv^T}{v\in \ker A}=\set {W\in \bS^n}{W=PWP\preceq  0}
$
\item[(c)]
$
\inter \cK_A= \set{V\in\Sn}{u^TVu> 0\;(u\in \ker A\setminus\{0\})}.
$

\end{itemize}
\end{proposition}

\noindent
The central result in Burke, Hoheisel and Gao \cite{BGH17} now follows.
\begin{theorem}[{\cite[Theorem 2]{BGH17}}]\label{th:clco D}
	Let $\cD(A, B)$ be  given by \eqref{eq:D}. Then
	\begin{equation*}\label{eq:omega}
	\clco \cD(A,B)=\Omega(A,B) :=\set{(Y,W)\in\bE}
	{AY=B \! \!\AND\!\! \half YY^T+W\in \cK_A^\circ}.	
	\end{equation*}
\end{theorem}

\noindent
In particular, Theorem \ref{th:clco D} 
implies that 
$
\vphi=\sigma_{\cD(A,B)}=\sigma_{\Omega(A,B)},
$
since $\sigma_S=\sigma_{\clco S}$ for all subsets $S$ of a Euclidean space.
This identity is used throughout.  

%
%

\section{Infimal projections of the generalized matrix-fractional function}\label{sec:Main}

\noindent
We now focus on infimal projections involving the GMF function. Consider  
\begin{equation}\label{eq:psi}
\psi:\bE\to\rbar,\quad \psi(X,V)=\varphi(X,V)+h(V), 
\end{equation} 
where  $\varphi\in \Gamma_0(\bE)$ is given in \eqref{eq:GMF support} and
$h\in \Gamma_0(\bS^n)$.
Our primary object of study is the infimal projection of the sum $\psi$  in the variable $V$  under the standing assumption that
$\rge B\subset \rge A$, i.e. $\set{Y\in\Rnm}{AY=B}\ne\emptyset$: 
\begin{equation}\label{eq:p}
p:\R^{n\times m}\to\rbar, \quad p(X)=\inf_{V\in\bS^n} \psi(X,V).
\end{equation}

\noindent
We lead with some elementary observations.

\begin{lemma}[Domain of $p$]\label{lem:p}  Let $p$ be defined by \eqref{eq:p}. Then the following hold: 

\begin{itemize}
\item[(a)] $p$ is convex.
\item[(b)] $\dom p=\set{X\in \R^{n\times m}}{\exists V\in \cK_A\cap \dom h: \;\rge \left(\begin{smallmatrix}  X\\B \end{smallmatrix}\right)\subset \rge M(V)}.$
In particular, $\dom p\neq \emptyset$ if and only if $\dom h\cap \cK_A\neq \emptyset$. 
\end{itemize}
Moreover, if $\dom p\neq \emptyset$ then the following hold: 
\begin{itemize}
\item[(c)]  If $B=0$ (e.g. if  $A=0$) then $\dom p$ is a subspace, hence relatively open.
\item[(d)] If $\rank A=\ell$ (full row rank) and 
$\dom h\cap\inter \cK_A\ne\emptyset$, then $\dom p = \R^{n\times m}$.
\item[(e)] If $\dom h\cap \cK_A\neq \emptyset$ and 
$(\dom h)^\infty\cap \cK_A= \{0\}$,
then $p$ is proper, hence $p\in \Gamma$.
\end{itemize} 
\end{lemma}
\begin{proof} (a) The convexity follows from, e.g.,  \cite[Proposition 2.22]{RoW98}.  
\smallskip

\noindent
(b)   We have $X\in \dom p$  if and only if there is a $V\in\Sn$ such that 
$(X,V)\in \dom \psi=(\dom\vphi)\cap (\Rnm\times \dom h)$. Hence the 
representation for $\dom p$ follows from 
the one of $\dom \vphi $ in  \eqref{eq:dom support D}.  
This representation for $\dom p$ tells us that $\dom p\neq \emptyset$ implies
that $\dom h\cap \cK_A\neq \emptyset$. On the other hand, if $V\in \dom h\cap \cK_A$,
then $(VY,V)\in \dom \psi$ for any $Y\in \Rnm$ satisfying $AY=B$, and so 
 $(VY,V)\in \dom p\ne \emptyset$.
 \smallskip

\noindent
(c)  If $B=0$, we have  $X\in \dom p $ if and only if $\lin \{X\}\subset \dom p.$     Since $\dom p$ is also convex, it is a subspace, see, e.g.,  \cite[Proposition 3.8]{RoW98}.
 \smallskip


\noindent
(d) By the description of  $\inter \cK_A$ in Proposition \ref{prop:K_A} (c), the assumptions imply that there exists $V\in\dom h\cap \cK_A$ such that $M(V)$ is invertible, see \cite[Proposition 3.3]{BuH15}. This readily gives the desired statement in view of (b).

\smallskip 

\noindent
(e) By part (b), $\dom p\neq \emptyset$. Hence let $X\in \dom p$, i.e. 
there is a $V\in \cK_A\cap \dom h$  such that  $ \rge \left(\begin{smallmatrix}  X\\B \end{smallmatrix}\right)\subset \rge M(V)$. 
If $p(X)=-\infty$,
there is a sequence $\{V_k\in \bS^n\cap\dom h\}$ with 
$\{(X,V_k)\in \dom\varphi\}$ such that $\psi(X,V_k)\goesto -\infty$. This implies that 
$\varphi(X,V_k)\to -\infty$ or $h(V_k)\to -\infty$. In either case, this tells
us that $\norm{V_k}\goesto \infty$ since both $\varphi$ and $h$ are closed  and proper. Consequently, there is a subsequence $J\subset\N$,
and a matrix $\widehat V\in\Sn$  such that
$(V_v/\norm{V_k})\overset {J}{\goesto}\widehat V$. Hence $0\neq \widehat V\in (\dom h\cap \cK_A)^\infty=(\dom h)^\infty\cap \cK_A$,
which contradicts the hypothesis.
\end{proof}

\noindent
We give two examples to illustrate various statements in Lemma  \ref{lem:p}. The first shows that an assumption of the type   in part (e) is required to establish that $p$ is proper.


\begin{example}[$p$ improper]\label{ex:p not proper} 
Let $m=n=1$, $A=0$, $B=0$ and $h(v)=-v$. Since $v^\dagger=\left\{\begin{array}{rcl} \frac{1}{v} & \text{if} & v\neq 0,\\
0 & \text{if} & v=0,
 \end{array}\right.$
 we have
$
\varphi(x,v)=\left\{\begin{array}{rcl} \frac{x^2}{2v} & \text{if} & v>0,\\
0 & \text{if} & v=0,\\
+\infty & \text{if} & v<0\\
 \end{array}\right.\quad ((x,v)\in \R^2).
$
Therefore, $p\equiv -\infty$ since
\[
p(x)=\inf_{v\in \R} \varphi(x,v)+h(v)=\inf_{v>0}\left\{ \frac{x^2}{2v}-v\right\}=-\infty \quad (x\in \R).
\]
\end{example}

\noindent
The properness condition given in  Lemma  \ref{lem:p} (e) is revisited in 
Definition \ref{def:CQ} where it
is called  \emph{boundedness primal constraint qualification} (BPCQ).
It is the strongest of the constraint qualifications we discuss. 

The second example shows $\dom p$ may not be relatively open 
if $B\neq 0$.

\begin{example}[$\dom p$  not relatively open]\label{ex:Counterex dom p} 
Let $A=\left(\begin{smallmatrix}1 & 1 \\ 1& 1\end{smallmatrix}\right)$ and $b=\left(\begin{smallmatrix}1 \\ 1\end{smallmatrix}\right)$. Then 
\(
\ker A=\lin \{\left(\begin{smallmatrix}1 \\ -1\end{smallmatrix}\right)\}
\) and \(\cK_A=\set{\left(\begin{smallmatrix}v& w \\ w& u\end{smallmatrix}\right)}{v+u\geq 2w}.
\)
Moreover, put $\bar V:=\left(\begin{smallmatrix} 2 & 1 \\ 1& 0\end{smallmatrix}\right)$  and define  
\(
\cV:=[0,1]\cdot \bar V=\set{\left(\begin{smallmatrix}2w& w \\ w& 0\end{smallmatrix}\right)}{w\in [0,1]}\subset \bS^2.
\)
Then $\cV$ is convex and compact. 
Let $h\in \Gamma_0(\bS^2)$ be any function with $\dom h=\cV$. Note that 
\(
\dom h\cap \cK_A=\cV.
\)
Hence
\begin{eqnarray*}
x\in \dom p & \Longleftrightarrow &\exists w\in [0,1]:  \; \left(\begin{smallmatrix} x \\ b\end{smallmatrix}\right)\in \rge \left(\begin{smallmatrix} w\bar V  & A^T \\
A& 0\end{smallmatrix}\right)\\
& \Longleftrightarrow & \exists w\in [0,1], r,s\in \R^2: \; \begin{array}{rcl} x& 
= &w\bar V r+A^Ts,\\
b & = & Ar
\end{array}\\
& \Longleftrightarrow & \exists w\in[0,1], \lambda, \mu \in \R:\; x=w\left(\begin{smallmatrix}2 & 1\\1& 0\end{smallmatrix}\right)\left[\left(\begin{smallmatrix}0 \\ 1\end{smallmatrix}\right)+\lambda \left(\begin{smallmatrix}1\\-1\end{smallmatrix}\right)\right]+\mu\left(\begin{smallmatrix}1\\1\end{smallmatrix}\right)\\
& \Longleftrightarrow & \exists w\in[0,1], \gamma\in \R:\;  x=w \left(\begin{smallmatrix}1\\0\end{smallmatrix}\right)+\gamma \left(\begin{smallmatrix}1\\1\end{smallmatrix}\right).
\end{eqnarray*}
Therefore,
\(
\dom p=[0,1]\cdot \left(\begin{smallmatrix}1\\0\end{smallmatrix}\right)+\lin \{\left(\begin{smallmatrix}1\\1\end{smallmatrix}\right)\},
\)
and hence 
\(
\ri(\dom p)=(0,1)\cdot \left(\begin{smallmatrix}1\\0\end{smallmatrix}\right)+\lin \{\left(\begin{smallmatrix}1\\1\end{smallmatrix}\right)\},
\)
so that $\dom p$ is clearly not relatively open. 
\end{example}

\noindent
The preceeding  example, shows that $\dom p$ may fail to be  a subspace if $B\neq 0$, hence  this   assumption in  Lemma  \ref{lem:p}(c)  is not superfluous. On the other hand,
 Lemma \ref{lem:p} (d) and  Example \ref{ex:Empty Subdiff} (a)  illustrate that  the condition $B=0$ is only sufficient but not necessary for $\dom p$ to be a subspace.

\subsection{The functions $\psi$, $\psi^*$, and their subdifferentials}\label{ss:psi}
The study of the infimal projection $p$ in \eqref{eq:p}  requires an 
understanding of the properties of the function $\psi$ from  \eqref{eq:psi}, 
its conjugate $\psi^*$, and their subdifferentials. For this we make extensive use of the condition
\begin{equation}\tag{CCQ}
\ri(\dom h)\cap \inter \cK_A\neq \emptyset,
\end{equation}
which we refer to as the \emph{conjugate constraint qualification}.  
As a direct consequence of the {\em line segment principle} (cf. \cite[Theorem 6.1]{RTR70}), we have
\begin{equation}\label{eq:simple ccq}
\ri(\dom h)\cap \inter \cK_A\neq \emptyset
\ \Longleftrightarrow\ 
\dom h\cap \inter \cK_A\neq \emptyset.
\end{equation}


\begin{lemma}[Conjugate of $\psi$]\label{lem:conj psi}
Let $\psi$ be given as in \eqref{eq:psi}  and define 
\begin{equation}\label{eq:conj psi}
\eta:(Y,W)\in \bE\mapsto \inf_{T\in\Sn}h^*(W-T)+\delta_{\Omega(A,B)}(Y,T).
\end{equation}
Then 
\begin{equation}\label{eq:dom eta}\begin{aligned}
\dom \eta&= \Omega(A,B)+(\{0\}\times \dom h^*)\\
&=\set{(Y,W)}{AY=B, \; \left(-\half YY^T+\cK_A^\circ\right)\cap (W-\dom h^*)\neq \emptyset },
\end{aligned}
\end{equation}
and the following hold:
\begin{itemize}
\item[(a)] If $\psi\not\equiv +\infty$, then $\psi\in\Gamma_0(\bE)$.
\item[(b)]  If $\dom{h}\cap \cK_A \ne\emptyset$  then $\psi, \psi^*\in\Gamma_0(\bE)$ with  $\psi^*=\cl \eta$.
\item[(c)]  Under CCQ , we have $\psi^*=\eta$. Moreover, in this case, the infimum 
in the definition of $\eta$
is attained on the whole domain, i.e.
\begin{equation}\label{eq:def T}
\kS(\barY, \barW)
:=\argmin_{T\in \Sn}\set{h^*(\barW\!-\!T)}{(\bar Y,T)\in\Omega(A,B)}\\
\end{equation}
is nonempty for all $(\bar Y,\bar W)\in \dom \psi^*$ .
\item[(d)]
Under CCQ,  $\dom\p \psi^*=\set{(Y,W)}{\emptyset\ne \kS(Y,W)}$ and, for
every $(Y,W)\in\dom \p \psi^*$, we have
\[
\p \psi^* (Y,W)\!=\!\set{(X,V)}{
\begin{aligned}&\exists\, T\in\bS^n:\;
V\in\p h^*(W-T)\cap\cK_A,
\\ &\ip{V}{\half YY^T+T}=0,
\;
\rge(X-VY)\!\subset\!(\ker A) ^\perp
\end{aligned}}.
\]
\end{itemize}
\end{lemma}

\begin{proof}  Note that $\eta(Y,W)<+\infty$ if and only if there is a $W_1,W_2\in\Sn$ such that
$W=W_1+W_2,\ (Y,W_1)\in\Omega(A,B)$ and $W_2\in\dom h^*$, or equivalently,
$(Y,W)\in\Omega(A,B)+(\{0\}\times \dom h^*)$, which in turn is equivalent to
$AY=B,\ T\in -\half YY^T+\cK_A^\circ$ and $T\in W-\dom h^*$ giving 
\eqref{eq:dom eta}.

Define $\map{\hat h}{\bE}{\rbar}$ by
$\hat h(X,V):= h(V)$. Then $\dom \hat h=\R^{n\times m}\times \dom h$ and
$\psi=\vphi +\hat h=\sig_{\Omega(A,B)}+\hat h$.
\smallskip
 
\noindent
(a) The sum of two closed, proper, convex functions (here $\vphi$ and $\hat h$) is closed and convex. It is proper if and (only) if  the sum is not constantly $+\infty$.
\smallskip

\noindent
(b) The sum of two proper functions is proper if and only if the domains of both functions intersect. By \eqref{eq:dom support D}, we have 
\(
\dom \hat h\cap \dom \vphi\neq \emptyset
\) 
if and only if
\(\dom h \cap \cK_A\neq \emptyset.
\)
Therefore, $\psi$ is proper if (and only if) the latter condition holds. Combined with (a) this shows $\psi\in\Gamma_0(\bE)$, 
and so $\psi^*\in\Gamma_0(\bE)$.
Moreover, by Appendix Theorem  \ref{prop:ExtSumRule General} (a),
\(
\psi^*(Y,W)=\cl\left(\del_{\Omega(A,B)}\infconv\hat h^*\right)(Y,W).
\)
Since  $\hat h^*(Y,W)=\del_{\{0\}}(Y)+h^*(W)$,  
\(
(\del_{\Omega(A,B)}\infconv\hat h^*)(Y,W)=\inf_{(Y,T)\in\Omega(A,B)}h^*(W-T)
=\eta(Y,W),
\)
proving $\psi^*=\cl \eta$.
\smallskip

\noindent
(c)  By \cite[Theorem 4.1]{BuH15}, $\inter (\dom \vphi)=\set{(X,V)}{V\in\inter \cK_A}$ and, by definition,
$\ri(\dom \hat h) = \R^{n\times m}\times \ri (\dom h)$.
Hence 
\begin{equation}\label{eq:CCQ applied}
\ri (\dom \hat h)\cap \ri(\dom \vphi)\neq \emptyset \quad \Longleftrightarrow \quad \ri(\dom h) \cap \inter \cK_A\ne\emptyset.
\end{equation}
Theorem \ref{prop:ExtSumRule General} (a) (applied to $\vphi$ and $\hat h$), CCQ, and \eqref{eq:CCQ applied}
imply $\psi^*=\eta$ with
\begin{equation}\label{eq:def T2}
\emptyset\ne\cT(\barY, \barW)
:=\!\!\!\!\argmin_{(Y,T),(0,W)\in\bE}
\set{h^*(W)}{(Y,T)\in\Omega(A,B),\ Y=\barY,\ \barW=W+T}.
\end{equation}
Since 
\begin{equation}\label{eq:T3}
\begin{aligned}
\kS(\barY, \barW)&=\set{T\in\Sn}{[(\barY,T),(0,\barW-T)]\in\cT(\barY, \barW)},\AND
\\
\cT(\barY, \barW)&=\set{[(\barY,T),(0,\barW-T)]}{T\in \kS(\barY, \barW)},
\end{aligned}
\end{equation}
we have $\kS(\barY, \barW)\ne\emptyset$ if and only if 
$\cT(\barY, \barW)\ne\emptyset$.
 \smallskip

\noindent
(d) Observe that  $\p \vphi^*=N_{\Omega(A,B)}$ and 
$\p \hat h^*=\R^{n\times m}\times \p h^*$ with $\dom \p \hat h^*=\{0\}\times \dom \p h^*$. 
Then  
part (c) and Theorem \ref{prop:ExtSumRule General} (d) (applied to $\vphi$ and $\hat h$) yield
\[
\begin{aligned}
\p \psi^*(Y,W)&=\set{(X,V)}{
\begin{array}{l}(X,V)\in\p \vphi^*(Y_1,W_1)\cap\p \hat h^*(Y_2,W_2),
\\
(Y,W)=(Y_1,W_1)+(Y_2,W_2)
\end{array}}
\\
&=\set{(X,V)}{\exists\, T\in\R^{n\times m}:\;
(X,V)\in N_{\Omega(A,B)}(Y,T),\ V\in \p h^*(W-T)}.
\end{aligned}
\]
The claim  follows from the representation  for 
$N_{\Omega(A,B)}(Y,T)$ in  \cite[Proposition 3]{BGH17}.
\end{proof}

%
%
%


\begin{corollary}[Subdifferential of $\psi$]\label{cor:Lem Subdiff psi}  
Let $\psi$ be given by \eqref{eq:psi} and $\kS$ by \eqref{eq:def T}.  Then the following hold:
\begin{enumerate}
\item[(a)]
If $(\bar Y,\bar W)\in \p \vphi(\bar X,\bar V)+(\{0\}\times \p h(\bar V))$, then
$\kS(\barY, \barW)\ne\emptyset$ and
\begin{equation}\label{eq:T sd equivalence}
\kS(\barY, \barW)=
\set{T\in\Sn}
{\barW-T\in \p h(\barV),\; 
(\barY,T)\in \p\vphi(\barX,\barV)}, 
\end{equation}
where $\p\vphi$ is described in \cite[Corollary 3.2]{BGH17}.
\item[(b)]
 Under CCQ we have 
\[
\dom \p\psi=\set{(X,V)}{V\in \dom \p h\cap \cK_A,\; \rge \binom{X}{B}\subset \rge M(V)}.
\]
Moreover,   for all $(\bar X,\bar V)\in \dom \p\psi$ and all $(\barY,\barW)\in\p\psi(\bar X,\bar V)$,  we have  $\kS(\barY,\barW)\ne\emptyset$
and
\begin{equation}\label{eq:Impl 1}
\begin{aligned}
\p\psi( \barX, \barV)&=\p \vphi(\bar X,\bar V)+(\{0\}\times \p h(\bar V))
\\ &=\set{(\barY,\barW)\in \bE}{\kS(\barY,\barW)\ne\emptyset}.
\end{aligned}
\end{equation}
\end{enumerate}
\end{corollary}
\begin{proof} 
Set $f_1(X,V):=\vphi(X, V)$ and $f_2(X,V):=h(V)$, so that  
the mapping $\cT$ in Theorem \ref{prop:ExtSumRule General} is given by
\eqref{eq:def T2}. Then, using \eqref{eq:T3}, part (a) follows from
Theorem \ref{prop:ExtSumRule General} (b), and part (b) follows from
Theorem \ref{prop:ExtSumRule General} (c).
\end{proof}

\subsection{Infimal projection I}
Let the infimal projection $p$ be as given in \eqref{eq:p}.
We are now in position to give a formula for $p^*$ under CCQ. 

\begin{theorem}[Conjugate of  $p$ and properties under CCQ] 
\label{th:inf-proj psi}
Let $p$ be given by \eqref{eq:p}.  Moreover, let $ \eta_0:\R^{n\times m}\to \rbar$ be given by 
 \begin{equation}\label{eq:eta0}
\eta_0:Y\mapsto \inf_{(Y,-W)\in\Omega(A,B)}h^*(W).
\end{equation}
Then the following hold:
\begin{itemize}
\item[(a)] $
\dom \eta_0=
\set{Y\in\R^{n\times m}}{AY=B,\;\left(-\frac{1}{2}YY^T+\cK^\circ_A\right)\cap (-\dom h^*)\neq \emptyset}$\\ 
$=\set{Y\in\R^{n\times m}}{(Y,0)\in \dom \eta},$ where $\eta$ is defined in 
\eqref{eq:conj psi}.
\item[(b)]
If $\dom{h}\cap \cK_A \ne\emptyset$, then $p^*=\cl \eta_0$, hence $\dom \eta_0\subset\dom p^*.$
\item[(c)] If CCQ holds for $p$,  then $\dom p=\R^{n\times m}$ and the following hold:
\begin{itemize}
\item[(I)]  $p^*=\eta_0$, i.e. 
\begin{equation}\label{eq:p conj}
p^*(Y)=\inf_{(Y,-W)\in\Omega(A,B)}h^*(W).
\end{equation}
Moreover, for all  $Y\in\dom p^*$, the infimum is a minimum, i.e.  there exists 
$W\in\dom h^*$ with $(Y,-W)\in\Omega(A,B)$ such that $p^*(Y)=h^*(W)$. 

In particular, $p^*$ is closed, proper convex under CCQ if and only if it is  proper, which is the case if and only if

\[
\begin{aligned}
\emptyset\ne \dom \psi^*(\cdot,0)
&=\set{Y}{\exists\, W\in\dom h^*:\;(Y,-W)\in \Omega(A,B)}
\\ &=\set{Y}{(Y,0)\in \Omega(A,B)+(\{0\}\times \dom h^*)},
\end{aligned}
\]
with  $\dom p^*=\dom \psi^*(\cdot,0) =\dom \eta_0$.

\item[(II)]  $p$  is either (convex) finite-valued 
(hence $p\in \Gamma_0(\Rnm)$) or  $p\equiv-\infty$. The former is the case if and only if   $\dom \psi^*(\cdot,0)\neq\emptyset$.
\end{itemize}


\end{itemize}

\end{theorem}

\begin{proof}
(a)   This follows from the definition of $\eta_0$. Also note that
$\eta_0=\eta(\cdot,0)$.

\smallskip

\noindent
(b) By Lemma \ref{lem:conj psi} (b), $\psi^*\in\Gamma_0(\bE)$ with $\psi^*=\cl\eta$
with $\eta$ defined in \eqref{eq:conj psi}.
Hence, by  \cite[Theorem 11.23 (c)]{RoW98}, $p^*=\psi^*(\cdot,0)$ which 
establishes the given representation. 
The domain containment is clear as $p^*=\cl \eta_0\leq \eta_0$.
\smallskip

\noindent
(c) Observe that 
\(
\dom p = L(\dom \vphi \cap \R^{n\times m}\times \dom h),
\)
where $L:(X,V)\mapsto X$, see Lemma \ref{lem:p}. 
By CCQ, we have $\ri (\dom h)\cap \inter \cK_A\neq \emptyset$, hence\begin{eqnarray*}
\ri(\dom \vphi \cap (\R^{n\times m}\times \dom h)) & = & \inter (\dom \vphi)\cap (\R^{n\times m}\times \ri (\dom h))\\
& = & (\R^{n\times m}\times \inter \cK_A)\cap (\R^{n\times m}\times \ri (\dom h))\\
& = & \R^{n\times m}\times (\inter \cK_A\cap \ri (\dom h)),
\end{eqnarray*}
where we use \cite[Theorem 4.1]{BuH15}  to represent 
$\inter (\dom \vphi)$. This now  gives
\[
\ri (\dom p)= L \left[\ri (\dom \vphi \cap \R^{n\times m}\times \dom h)\right]=\R^{n\times m}.
\]
\noindent
(c.I)  As in part (b),  $p^*=\psi^*(\cdot,0)$. Hence,  Lemma \ref{lem:conj psi} (c)  gives the identity $p^*=\eta_0$ under CCQ as well as the attainment statement. 
Since $\psi^*$ is closed, proper, convex (under CCQ) by 
Lemma \ref{lem:conj psi} (b), $\psi^*(\cdot,0)$ is, too,  if and only if $\dom \psi^*(\cdot,0)\ne \emptyset$, and so the statements about 
$p^*=\psi^*(\cdot,0)$ follow .
\\
\vskip .1in
\noindent
(c.II)  We have $\dom p=\R^{n\times m}$. By \cite[Corollary 7.2.3]{RTR70} this implies that either
$p\equiv-\infty$ or $p$ is finite-valued, which shows the first statement.  For the second,  again as $\dom p=\R^{n\times m}$, observe that the convex function $p$ is finite-valued if and only if it is proper, which is true if and only if $p^*$ is proper, so I) gives the desired statement.
\end{proof}

\noindent
Observe that Example \ref{ex:p not proper} shows that the condition 
$\emptyset\ne\dom\psi^*(\cdot,0)$ is essential
in Theorem \ref{th:inf-proj psi} (c.I-c.II). Indeed, in this example,
$p\equiv -\infty$ so $\dom p=\R$, while $h=\sig_{\{-1\}}$ and $h^*=\delta_{\{-1\}}$,
$\psi^*(\cdot,0)=p^*\equiv\infty$, and CCQ is satisfied.

We now broaden our perspective of infimal projection by embedding it into a {\em pertubation duality  framework} in the sense of  \cite[Theorem 11.39]{RoW98} or the development in \cite[Chapter 5]{AuT03}. 
Given $\bar X\in\R^{n\times m}$, define $f_{\bar X}$ by
\[
f_{\bar X}(X,V):= \psi(X+\bar X,V)\quad ((X,V)\in \bE),
\]
and $p_{\bar X}$ by 
\begin{equation}
p_{\bar X}(X):=\inf_{V\in \bS^n} f_{\bar X}(X,V)\quad (X\in \R^{n\times m}).
\end{equation}
Then 
\(
f_{\bar X}^*(Y,W)=\psi^*(Y,W)-\ip{\bar X}{Y}\quad ((Y,W)\in \bE),
\)
\cite[Equation 11(3)]{RoW98}.  Define
\begin{equation}\label{eq:qx}
q_{\bar X}(W):=-\sup_{Y}\{\ip{\bar X}{Y}-\psi^*(Y,W)\}  \quad (W\in\bS^n).
\end{equation}
Then $q_{\bar X}$ is a convex  function that pairs in duality with 
$p_{\bar X}$  satisfying the weak duality
\(
p_{\bar X}(0)\geq -q_{\bar X}(0)\quad (\bar X\in \R^{n\times m}).
\)
Applying the general pertubation duality to our scenario yields the following result.


\begin{proposition}[Shifted duality for $p$]\label{prop:Shifted duality}  Let $p$ be defined by \eqref{eq:p}, let $\bar X\in \dom p$  and $q_{\bar X}$ be defined by \eqref{eq:qx}. Then the following hold:
\begin{itemize}
\item[(a)] If $0\in \ri(\dom q_{\bar X})$ then $p(\bar X)=-q_{\bar X}(0)\in \R$, 
$\argmin \psi(\bar X,\cdot)\neq\emptyset$, and $\p q_{\bar X}(0)\neq\emptyset$. 

\item[(b)] If $\bar X\in \ri(\dom p)$ then $p(\bar X)=-q_{\bar X}(0)\in \R$, $\argmax_{Y}\{ \ip{\bar X}{Y}-\psi^*(Y,W) \}\neq \emptyset$, and  $\p p(\bar X)\neq \emptyset $.

\item[(c)] Under either condition $0\in \ri(\dom q_{\bar X})$ or $\bar X\in \ri(\dom p)$, $p$ is lsc at $\bar X$and $-q_{\bar X}$ is lsc at $0$.

\item[(d)] We have 
\[
\left.\begin{array}{rcl}
& & p(\bar X)\\& = & \!\!\psi(\bar X,\bar V),\\
& = &\!\!\ip{\bar X}{\bar Y}- \psi^*(\bar Y,0),\\
& =  &\!\!-q_{\bar X}(0) 
\end{array}\right\}\!\!\iff\!\!  (\bar Y,0)\!\in\! \p\psi(\bar X,\bar V) \!\!\iff\!\! (\bar X,\bar V) \!\in\! \p\psi^*(\bar Y,0).
\]
\end{itemize}
\end{proposition}

\begin{proof}  Let $\bar X\in \dom p$ and observe that 
 \(
 p(X+\bar X)=p_{\bar X}(X)\quad (X\in \bR^{n\times m}),
 \)
 hence,  in particular, $p(\bar X)=p_{\bar X}(0)\in \R$. Moreover, notice that $\psi$ and hence $f_{\bar X}$ is proper (hence in $\Gamma_0$) as by assumption $\bar X\in \dom p$ exists. Applying the results
 \cite[Theorem 5.1.2--5.1.5, Corollary 5.1.2]{AuT03} to the duality 
 pair $p_{\bar X}$ and $q_{\bar X}$ and translating from $p_{\bar X}$ 
 at $0$ to $p$ at $\bar X$ gives all the desired statements. 
\end{proof}

\noindent
The domain of $q_{\bar X}$ plays a key role in interpreting this result
in a given setting. Below we provide a useful representation of this domain 
using the set
\begin{equation}\label{eq:C}
\Omega_2(A,B):=\set{W\in\bS^n}{\exists Y: \; (Y,W)\in \Omega(A,B)}.
\end{equation}

\begin{lemma}[Domain of $q_{\bar X}$]\label{lem:PCQ and duality} Let $\bar X\in \R^{n\times m}$ and $q_{\bar X}$ defined by \eqref{eq:qx}. Then
\(
\dom q_{\bar X}=\Omega_2(A,B)+\dom h^*.
\)
\end{lemma}
\begin{proof}  Using Lemma \ref{lem:conj psi}, observe that
\begin{eqnarray*}
q_{\bar X}(W)&=& \inf_{Y} \left\{\psi^*(Y,W)-  \ip{\bar X}{Y} \right\}\\
& = & \inf_{Y} \left\{\eta(Y,W)-  \ip{\bar X}{Y} \right\}\\
& = & \inf_{(Y,T)\in \Omega(A,B)} \left\{ h^*(W-T)-  \ip{\bar X}{Y} \right\}.
\end{eqnarray*}
Therefore,
\[
\dom q_{\bar X}= \set{W\in \bS^n}{\exists (Y,T)\in \Omega(A,B): W-T\in \dom h^* }=  \Omega_2(A,B)+\dom h^*.
\]
\end{proof}

\noindent
We now discuss various constraint qualifications for $p$.

\subsection{Constraint qualifications}\label{sec:CQ} 

We start our analysis with a result about the set $\Omega_2(A,B)$ from \eqref{eq:C}, which was used  in Lemma \ref{lem:PCQ and duality}
to represent the domain of $q_{\bar X}$.

\begin{lemma}[Properties of $\Omega_2(A,B)$]\label{lem:C}  Let $\Omega_2(A,B)$ be as in \eqref{eq:C}.  Then we have:
\begin{itemize}
\item[(a)] $\Omega_2(A,B)$ is closed and convex with $\Omega_2(A,B)^\infty=\cK_A^\circ$.
\item[(b)] $\Omega_2(A,B) =\dom \vphi(\bar X,\cdot)^* $ for all 
$\bar X\in \Rnm$ such that $\vphi(\bar X,\cdot)$ is proper. 
\item[(c)] We have 

\centerline{$
\ri \Omega_2(A,B) \!\! = \!\!  \set{W\!}{\!\exists Y:\; AY\!=\!B,\; 
\frac{1}{2}YY^T\!+\!W\in \ri (\cK_A^\circ)\!}
 \!\! = \!  \ri(\dom \vphi(\bar X,\cdot)^*)
$}
for all $\bar X$ such that $\vphi(\bar X,\cdot)$ is proper. 

\end{itemize}

\end{lemma}
\begin{proof}  (a) With the linear map $T:(Y,W)\mapsto W$ we have $\Omega_2(A,B)=T(\Omega(A,B))$.  Therefore $\Omega_2(A,B)$ is  convex.   
By \cite[Proposition 10]{BGH17}, we have $\Omega(A,B)^\infty=\{0\}\times \cK^\circ_A$, and so
$\ker T\cap \Omega(A,B)^\infty =\{0\}$ giving the 
remainder of (a) by \cite[Theorem 3.10]{RoW98}. 
\smallskip

\noindent
(b) Recall from \cite[Theorem 4.1]{BuH15} that $\inter (\dom \sigma_{\vphi})=\set{(X,V)\in\bE}{V\in \inter \cK_A}$, Thus we can apply 
Proposition \ref{prop:Partial Conj} to $\bar g:=\vphi(\bar X,\cdot)$ to infer that 
\[
\bar g^*(W)=\inf_{Y:(Y,W)\in \Omega(A,B)}\ip{-\bar X}{Y}\quad (W\in \bS^n).
\] 
This proves the claim. 
\smallskip

\noindent
 (c) Observe that $\ri \Omega_2(A,B)=\ri T(\Omega(A,B))=T(\ri \Omega(A,B))$  and use \cite[Proposition 8]{BGH17} to get the first representation. The second one follows from (b). 
\end{proof}

\noindent
We now define the constraint qualifications  central to our study.  
Note that CCQ was previously introduced in Section \ref{ss:psi}. 

\begin{definition}[Constraint qualifications]\label{def:CQ} Let $p$ be given by \eqref{eq:p}. We say that $p$ satisfies
\begin{itemize}
\item[(i)]   {\em PCQ:}   if $0\in \ri(\Omega_2(A,B) +\dom h^*)$;
\item[(ii)]  {\em strong PCQ (SPCQ):} if 
$0\in \inter (\Omega_2(A,B) +\dom h^*)$;
\item[(iii)]{\em boundedness  PCQ (BPCQ):} if $\dom h\cap \cK_A\neq \emptyset $ and $(\dom h)^\infty\cap\cK_A=\{0\}$;
\item[(iv)] {\em CCQ:} if $\ri(\dom h)\cap \inter \cK_A\neq \emptyset $.
\item[(v)] {\em strong CCQ (SCCQ):} if CCQ is satisfied and $\emptyset\ne\dom \psi^*(\cdot,0)$,
or equivalently,
\begin{equation}\label{eq:sccq}
\begin{aligned}
\emptyset\ne \Xi(A,B)&:=\set{ Y\in \Rnm}{  
AY=B,\;\half YY^T\in \dom h^*+\cK_A^\circ}
\\ &=
\set{Y\in\Rnm}{(Y,0)\in\Omega(A,B)+(\{0\} \times \dom h^*)}.
\end{aligned}
\end{equation}
\end{itemize}
\end{definition}

\noindent
The notation PCQ stands for {\em primal constraint qualification} while CCQ 
stands for {\em conjugate constraint qualification}. 
Theorem \ref{th:inf-proj psi}  and  Lemma \ref{lem:PCQ and duality}, respectively, 
give the following useful implications:
\begin{equation}\label{eq:Xi}
\begin{aligned}
\mathrm{CCQ} &\implies
\dom p^* = \Xi(A,B)
\\
\mathrm{SCCQ}&\implies
\dom p^* = \Xi(A,B) \ne \emptyset.
\end{aligned}
\end{equation}

The following results clarify the relations between the various constraint qualifications.  
We lead with characterizations of PCQ and BPCQ. 

\begin{lemma}[Characterizations of (B)PCQ]\label{lem:PCQ Equiv} Let $p$ be given by \eqref{eq:p} and  $\bar X\in \dom p$, and set 
\begin{equation}
\psi_{\bar X}:= \psi(\bar X,\cdot)\quad (\bar X\in \R^{n\times m}).
\end{equation}
 \begin{itemize}
 \item[(a)]The  following are equivalent: 
\begin{itemize}
\item [(i)] $0\in \ri (\dom \psi_{\bar X}^*)$;
\item[(ii)] PCQ holds for $p$;
\item[(iii)] $\exists\, Y\in \R^{n\times m}:\; AY=B,\quad \frac{1}{2}YY^T\in \ri(\cK^\circ_A+\dom h^*)$.
\end{itemize}
In addition, similar characterizations of SPCQ hold
by substituting the interior for the relative interior.
\item[(b)]  BPCQ holds for $p$ if and only if $ \dom h\cap \cK_A$ is nonempty and bounded.
\end{itemize}
\end{lemma}

\begin{proof} 
(a) Defining $\vphi_{\bar X}:=\vphi(\bar X,\cdot)$,  we find that 
$\vphi_{\bar X}^*=\cl(\vphi_{\bar X}^*\infconv h^*)$ and therefore 
$\ri(\dom \psi_{\bar X}^*)=\ri(\dom \vphi_{\bar X}^*+\dom h^*)=\ri (\Omega_2(A,B)+\dom h^*)$, see Lemma \ref{lem:C} (c).   This proves the first two equivalences. The third follows readily from the representation of $\ri(\Omega(A,B))$ from \cite[Proposition 8]{BGH17}.

\smallskip

\noindent
(b) Follows readily from \cite[Theorem 3.5, Proposition 3.9]{RoW98}.
\end{proof}

\noindent
We point out that, under PCQ,  Lemma \ref{lem:PCQ Equiv}  shows that the objective functions $\psi(\bar X,\cdot)\;(\bar X\in \dom p)$  occuring in the definition of $p$ in \eqref{eq:p} are {\em weakly coercive} 
\cite[Definition 3.2.1]{AuT03} when proper, see \cite[Theorem 3.2.1]{AuT03}.  This tells us that the infimum in \eqref{eq:p} is attained under PCQ if finite
\cite[Proposition 3.2.2, Theorem 3.4.1]{AuT03}, a fact that is stated again 
(and derived alternatively) in Theorem \ref{th:p under PCQ}.  Under SPCQ, the objective functions $\psi(\bar X,\cdot)\;(\bar X\in \dom p)$ are {\em level-bounded} (or {\em coercive}), in which case the $\argmin \psi(\bar X,\cdot)$ is nonempty and compact (and clearly convex).  
Finally, it was shown in Lemma \ref{lem:p} (e) that $p$ is closed proper convex
under BPCQ.

The next result shows the relations between the different notions of PCQ.

\begin{lemma}\label{lem:PCQ Implication} Let $p$ be given by \eqref{eq:p}. Then the following hold:
\begin{itemize}
\item[(a)]  BPCQ  \quad  $\Longrightarrow$\quad  SPCQ \quad $\Longrightarrow$ \quad PCQ.
\item[(b)] If $\inter (\dom h^*)\cap \inter(- \Omega_2(A,B))\neq \emptyset $, then PCQ and SPCQ are equivalent.
\end{itemize}
\end{lemma}
\begin{proof} (a) The first implication can be seen as follows:  If BPCQ holds then $\dom \psi_{\bar X}\subset \dom h\cap \cK_A$ is bounded  (and nonempty exactly if $\bar X\in \dom p$). Therefore $\psi_{\bar X}$ is level-bounded for all $\bar X\in \dom p$, i.e. 
$0\in \inter (\dom \psi_{\bar X}^*)\;(\bar X\in \dom p)$, 
see e.g. \cite[Theorem 11.8]{RoW98}. In view of Lemma \ref{lem:PCQ Equiv} (a) this implies that SPCQ holds.  

The second implication is trivial.

\smallskip

\noindent
(b) This is follows directly from the definitions.
\end{proof}


\noindent
We now provide characterizations for CCQ. 

\begin{lemma}[Characterizations of CCQ]\label{lem:CCQ equiv} Let $p$ be given by \eqref{eq:p}. Then 
\begin{center}
(i)\;   $\dom{h}\cap\inter \cK_A \ne\emptyset$  $\iff$ (ii)\;CCQ holds for $p$ $\iff$ (iii)\; $ (-\cK_A^\circ)\cap \hzn h^*=\{0\}$. 
\end{center}
\end{lemma}

\begin{proof}
The first equivalence was previously observed in \eqref{eq:simple ccq}.
%
The second equivalence can be seen as follows: 
We apply \cite[Corollary 16.2.2]{RTR70} (to $f_1:=h$ and $f_2:=\delta_{\cK_A}$).
This result tells us that
$\ri(\dom{h})\cap\inter \cK_A \ne\emptyset$ if and only if there does not exist
a matrix $W\in \Sn$ such that
\begin{equation}\label{eq:cq equivalence}
(h^*)^\infty(W)+\support{-W}{\cK_A}\le 0\AND
(h^*)^\infty(-W)+\support{W}{\cK_A}> 0.
\end{equation}
Since $\support{-W}{\cK_A}=\indicator{-W}{\cK_A^\circ}$, the first of these conditions
is equivalent to the condition $W\in (-\cK_A^\circ)\cap\hzn h^*$.  In particular, we can infer that $(-\cK_A^\circ)\cap\hzn h^*=\{0\}$ gives the inconsistency of  \eqref{eq:cq equivalence} and thus establishes  (iii)$\Rightarrow$(ii).

The second condition in \eqref{eq:cq equivalence} implies  $W\neq 0$. Thus, 
in view of Proposition \ref{prop:K_A} (b), $0\neq -W\in\cK_A^\circ\subset\Sn_+$, and hence 
$W\notin \cK_A^\circ$. Thus, every nonzero element
of the set $(-\cK_A^\circ)\cap\hzn h^*$ satisfies \eqref{eq:cq equivalence}.
Thus, the nonexistence of a $W$ satisfying \eqref{eq:cq equivalence}
implies  that $(-\cK_A^\circ)\cap\hzn h^*=\{0\}$, which altogether  proves the result.
\end{proof}

\noindent
 Note that for any proper, convex function $f$ we always 
have $\hzn f\subset (\dom f)^\infty$ which, in view of Lemma \ref{lem:CCQ equiv}, implies that the condition
\begin{equation}\label{eq:SCCQ}
(-\cK_A^\circ)\cap (\dom h^*)^\infty=\{0\}
\end{equation}
is stronger than CCQ. However, \eqref{eq:SCCQ} is not used in our study. 

\subsection{Infimal projection II} We  return to our analysis of the infimal  projection defining $p$ in \eqref{eq:p}. 
The following result shows that the two key conditions 
appearing in Proposition \ref{prop:Shifted duality},
$0\in \ri(\dom q_{\bar X})$ and  $\bar X\in \ri(\dom p)$,  
correspond nicely to the constraint qualifications studied in Section \ref{sec:CQ}.

\begin{corollary} \label{cor:CQ duality} Let $p$ be defined by \eqref{eq:p}, let $\bar X\in \dom p$ and $q_{\bar X}$ be defined by \eqref{eq:qx}. Then the following hold:
\begin{itemize}
\item[(a)] PCQ holds for $p$ if and only if $0\in \ri(\dom q_{\bar X})$;
\item[(b)] If CCQ holds, then  $\bar X\in \ri(\dom p)$.
\end{itemize}
\end{corollary}
\begin{proof} (a) Follows immediately from Lemma \ref{lem:PCQ and duality} and the definition of PCQ.
\smallskip

\noindent
(b) Under CCQ we have $\dom p=\R^{n\times m}$ 
(see the proof of Theorem \ref{th:inf-proj psi} (c.II)),  hence (b) follows.
\end{proof}

\noindent
As a consequence of Corollary \ref{cor:CQ duality} and Proposition \ref{prop:Shifted duality} we can add to the properties of $p$ proven in Theorem \ref{th:inf-proj psi}. 

\begin{theorem}[Properties of $p$ under PCQ]\label{th:p under PCQ} Let $p$, defined in \eqref{eq:p}, be such that PCQ is satisfied
and $\dom h\cap \cK_A\neq \emptyset$ (i.e. $\dom p\neq \emptyset$). 
Let $q_{\bar X}$ be given by \eqref{eq:qx}. Then the following hold:
\begin{itemize}
\item[(a)] $p\in \Gamma_0(\R^{n\times m})$;
\item[(b)]  $ \argmin_{V}\psi(\bar X, V)\neq \emptyset \quad (\bar X\in\dom p)$  \quad (primal attainment);
\item[(c)] $p(\bar X)=-q_{\bar X}(0)\quad (\bar X\in \dom p)$\quad 
(zero duality gap).
\end{itemize}
\end{theorem}
\begin{proof} Let $\bar X\in \dom p$.  Under PCQ, by Corollary \ref{cor:CQ duality},  we have $0\in \ri(\dom q_{\bar X})$. 
Hence, by Proposition \ref{prop:Shifted duality} (a), 
there is a $\bar V\in\Sn$ such that $p(\bar X)=\psi(\bar X,\bar V)$, and so,
by Proposition \ref{prop:Shifted duality} (c),
$p$ is lsc at $\bar X$ with $p(\bar x)\in \R$.  The discussion in \cite[p.~153]{AuT03} tells us that $p$ is, in fact, closed, proper, convex.

Finally, the equality $p(\bar X)=-q_{\bar X}(0)$, also follows from 
Proposition \ref{prop:Shifted duality} (a).
\end{proof}

\noindent
Theorem \ref{th:p under PCQ} can be proven entirely without 
the shifted duality framework in Proposition \ref{prop:Shifted duality} by
using the linear projection $L:(X,V)\to X$ used implicitly throughout our study.  It can be seen that $p=L\psi$ is a {\em linear image} in the sense described in \cite[p.~38]{RTR70}.  Then \cite[Theorem 9.2]{RTR70} gives all statements from Proposition \ref{th:p under PCQ} 
after realizing that the constraint qualification in \cite[Theorem 9.2]{RTR70}, 
which reads 
\begin{equation}\label{eq:LinIm CQ}
\psi(0,V)>0 \quad \text{or}\quad \psi^\infty(0,-V)\leq 0\quad (V\in\bS^n),
\end{equation}
since $\ker L=\{0\}\times \bS^n$,  
is equivalent to PCQ in this setting.
However,
we chose to derive Theorem \ref{th:p under PCQ} from the shifted duality scheme since this assists in the subdifferential analysis.

The next result follows readily from the foregoing study.

\begin{corollary}\label{cor:P+C=S}  Let $p$ be given by \eqref{eq:p}
and $\eta_0$ by \eqref{eq:eta0}. 
If PCQ and CCQ  are satisfied for $p$ then the following hold: 
\begin{itemize}
\item[(a)] SCCQ holds and $p$ is finite-valued.
\item[(b)] (primal attainment) $p\in \Gamma_0(\R^{n\times m})$ is finite-valued and for all $\bar X\in \R^{n\times m}$ there exists $\bar V$ such that $p(\bar X)=\psi(\bar X,\bar V)$.
\item[(c)] (dual attainment) $p^*=\eta_0$ and for all $\bar Y\in \dom p^*$ there exists $\bar W$ such that $(\bar Y,\bar W)\in \Omega(A,B)$ and $p^*(\bar Y)=h^*(-\bar W)$.

\end{itemize}
\end{corollary}
\begin{proof} (a) Follows readily from Lemma \ref{lem:PCQ Equiv} a) and the definition of SCCQ. 
\smallskip

\noindent
(b) By (a), SCCQ holds, so the first statement follows from  Theorem \ref{th:inf-proj psi} (c). The second is due to Theorem \ref{th:p under PCQ} (b).
\smallskip

\noindent
(c) Since SCCQ holds, see (b),  Theorem \ref{th:p under PCQ} (c) applies.

\end{proof}

\noindent
The table below   summarizes most of  our findings so far. Here $\bar X\in\dom p$. 
\\

\begin{center}
\begin{tabular}{|c|c|c|c|c|c|c|}
\hline
Consequence$\backslash$Hypoth.& PCQ & SPCQ  & BPCQ & CCQ & SCCQ &  PCQ+CCQ\\
\hline
$p\in \Gamma_0\vee p\equiv -\infty$ & \checkmark & \checkmark& \checkmark& \checkmark& \checkmark& \checkmark \\
\hline
$p\in \Gamma_0$ & \checkmark & \checkmark& \checkmark& & \checkmark& \checkmark \\

\hline
$p(\bar X)=-q_{\bar X}(0)$ & \checkmark & \checkmark& \checkmark&\checkmark\footnotemark[1] &\checkmark & \checkmark  \\
\hline 
$\argmin \psi(\bar X,\cdot)\neq \emptyset$ & \checkmark & \checkmark& \checkmark& & & \checkmark   \\
\hline
$\argmin \psi(\bar X,\cdot)$ compact & & \checkmark & \checkmark\footnotemark[2] & & &   \checkmark  \\
\hline
$\dom p=\R^{n\times m}$  & & & &   \checkmark &   \checkmark&   \checkmark\\
\hline
$\displaystyle\argmin_{(\bar Y,T)\in \Omega(A,B)} \hspace{-3mm}h^*(-T)\neq \emptyset$ & & & &   \checkmark &   \checkmark&   \checkmark\\ 
\hline 
\end{tabular}
\footnotetext[1]{$p(\bar X)\equiv -\infty$ is possible.}
\footnotetext[2]{BPCQ also implies that $\dom \psi(\bar X,\cdot)$ is bounded.}
\end{center}
\label{tab:Overview}

\vspace{0.3cm}

%

\noindent
In view of Proposition \ref{prop:Shifted duality} (b) and Corollary \ref{cor:CQ duality} one might be inclined to think that using CCQ instead of the pointwise condition $\bar X\in \ri(\dom p)$ is excessively strong. 
However, computing the relative interior of $\dom p$ without CCQ is problematic, cf. the derivations in the proof of Theorem \ref{th:inf-proj psi} (c) under CCQ.   Hence, we do not 
consider constraint qualifications  weaker than CCQ.

We now turn our attention to subdifferentiation of $p$.

\begin{proposition}[Subdifferential of $p$]\label{prop:SD + PS} Let $p$ be given by \eqref{eq:p}. Then the following hold:

\begin{itemize}

\item[(a)]  Under SCCQ, $\dom p=\dom \p p=\Rnm$ and we have 
\begin{equation}\label{eq:SD p basic}
\p p(\bar X) = \argmax_{Y}\{ \ip{\bar X}{Y}-\inf_{(Y,T)\in \Omega(A,B)}h^*(-T)\},
\end{equation}
which is nonempty and compact.

\item[(b)] Under PCQ  equation \eqref{eq:SD p basic} holds,  and, for $\bar X\in \dom p$,  we have 
\begin{eqnarray*}
\p p(\bar X) &= & \set{\bar Y}{\exists \bar V:\; (\bar Y,0)\in \p \psi(\bar X,\bar V)}\\
& = & \set{\bar Y}{\exists \bar V:\; (\bar X,\bar V)\in \p \psi^*(\bar Y,0)}\\
&= & \set{\bar Y}{\exists \bar V: p(\bar X)=\psi(\bar X,\bar V)=\ip{\bar X}{\bar Y}-p^*(\bar Y)}.
\end{eqnarray*}

\item[(c)]  Under  PCQ and CCQ,  $\dom p=\dom \p p=\Rnm$ and we have 
\[
\p p(\bar X) = \set{Y}{\exists \bar V, \bar T: \; -\bar T\in \p h(\bar V),\; (Y,\bar T)\in \p \vphi(\bar X,\bar V)},
\]
which is nonempty and compact.
\end{itemize}

\end{proposition}
\begin{proof}   (a) Under SCCQ,  $p$ is  convex and finite-valued (hence closed and proper), therefore  $\dom p=\dom \p p=\Rnm$ with $\p p(X)$ compact for all $X\in \Rnm$.
The representation \eqref{eq:SD p basic} 
follows from \cite[Theorem 23.5]{RTR70} and the fact that the closure  
for $p^*$ can be dropped in the argmax problem.  
\smallskip

\noindent  
(b) Under PCQ we also have that $p\in \Gamma_0$, hence the same reasoning as in (a) gives \eqref{eq:SD p basic}. We now prove the remainder: 
For the first  identity   notice that (see e.g.  \cite[Chapter D,  Corollary 4.5.3]{HUL 01})
\[
\p p(\bar X)=\set{Y}{(Y,0)\in \p \psi(\bar X,\bar V)} \quad (\bar V\in \argmin_{V} \psi(\bar X,V)),
\] 
the latter argmin set being nonempty due to what was argued above.  The '$\subset$'-inclusion is hence clear.  For the reverse inclusion invoke the results
in \cite[Example 10.12]{RoW98} to see that if $(Y,0) \in \psi(\bar X,\bar V)$ then $\bar V\in  \argmin_{V} \psi(\bar X,V)$. 

The second identity in (c) is clear from \cite[Theorem 23.5]{RTR70} as $\psi\in \Gamma_0(\bE)$.

The third  follows from Proposition \ref{prop:Shifted duality}  in combination with  Corollary \ref{cor:CQ duality} 
and recalling that  $\psi^*(\bar Y,0)=p^*(\bar Y)$.  
\smallskip

\noindent
(c)  Apply  Corollary  \ref{cor:Lem Subdiff psi} to the first representation in (b).  
\end{proof}

\noindent
For $\bar X\in \rbd (\dom p)$ the subdifferential $\p p(\bar X)$ can be empty.  Moreover, it is unbounded if   $\bar X\notin \inter(\dom p)$.  The latter may even occur under BPCQ as the following example shows.


\begin{example}\label{ex:Empty Subdiff} 
Let $A=\left(\begin{smallmatrix}1 & 0 \\ 0& 0\end{smallmatrix}\right)$ and $b=\left(\begin{smallmatrix}1  \\ 0\end{smallmatrix}\right)$ so that 
\(
\cK_A=\set{\left(\begin{smallmatrix}v & w \\ w& u\end{smallmatrix}\right)}{u\geq 0}.
\)
Defining $h:=\delta_\cV$ for 
\(
\cV:=\set{\left(\begin{smallmatrix}v & 0 \\ 0&u\end{smallmatrix}\right)}{u\leq 0, v\in [0,1]}
\)
we hence find that
\(
\dom h\cap \cK_A=\set{\left(\begin{smallmatrix}v & 0 \\ 0& 0\end{smallmatrix}\right)}{v\in [0,1]}
\)
and
\(\dom h\cap \inter \cK_A=\emptyset,
\)
so that CCQ is violated but BPCQ (hence (S)PCQ) holds.  We find that 
\begin{eqnarray*}
x\in \dom p & \Longleftrightarrow & \exists\, V\in \cV\cap \cK_A:\; \left(\begin{smallmatrix} x\\ b\end{smallmatrix}\right)\in \rge \left(\begin{smallmatrix}V & A^T \\ A& 0\end{smallmatrix}\right)\\
&\Longleftrightarrow & \exists\, v\in [0,1], r,s\in \R^2:\;\begin{array}{rcl}
x& = & \left(\begin{smallmatrix} v & 0\\ 0 & 0\end{smallmatrix}\right)r +\left(\begin{smallmatrix} 1& 0\\ 0 & 0\end{smallmatrix}\right)s,\\
\left(\begin{smallmatrix} 1\\ 0\end{smallmatrix}\right) & = & \left(\begin{smallmatrix} 1& 0\\  0 & 0\end{smallmatrix}\right)r
\end{array}\\
& \Longleftrightarrow & \exists\, v\in [0,1], \rho, \sigma\in \R: x=  \left(\begin{smallmatrix} v & 0\\ 0 & 0\end{smallmatrix}\right)\left[ \left(\begin{smallmatrix}1\\ 0\end{smallmatrix}\right)+\rho  \left(\begin{smallmatrix} 1 \\ 0 \end{smallmatrix}\right)\right]+\sigma \left(\begin{smallmatrix} 1\\ 0 \end{smallmatrix}\right)\\
& \Longleftrightarrow & x\in \lin \{\left(\begin{smallmatrix} 1\\ 0 \end{smallmatrix}\right)\}.
\end{eqnarray*}
Therefore we have 
$
\dom p= \lin \{\left(\begin{smallmatrix} 1\\ 0 \end{smallmatrix}\right)\}.
$
In particular, $\dom p$ is a proper subspace of $\R^2$, hence  relatively open  with  empty interior. Therefore $\p p(x)$ is nonempty and  unbounded for any $x\in \dom p$.

\end{example}

\section{Infimal projection with a support function}\label{sec:Support Func} 
\noindent
We now study the case where $h$ is a support function: 
\begin{equation}\label{eq:p Support Func Case}
p(X):=\inf_{V\in \bS^n} \vphi(X,V)+\sigma_{\cV}(V),
\end{equation}
where $\cV$ is a given closed, convex  subset of $\bS^n$.
Our first task is to interpret the 
constraint qualifications of Section \ref{sec:CQ} when $h=\sig_\cV$.
Here, and for the remainder of this section, the choice 
$h=\sigma_{\cV}$ implies that $\dom h=\barr\cV$ and $\dom h^*=\cV$.

\begin{lemma}[Constraint qualifications for \eqref{eq:p Support Func Case}]\label{lem: CQ support} Let $p$ be given by \eqref{eq:p Support Func Case}.  Then the  following hold:
\begin{itemize}
\item [(a)] (CCQ) The conditions 
\begin{eqnarray}
\barrier \cV\cap\inter \cK_A\neq \emptyset, \label{eq:CQ support 1a}\\
\cV^\infty\cap (-\cK_A^\circ)=\{0\}, \label{eq:CQ support 1b}\\
\cl(\barrier \cV)-\cK_A  = \bS^n \label{eq:CQ support 1c}
\end{eqnarray}
are each equivalent to CCQ for $p$ in \eqref{eq:p Support Func Case}. 
Moreover, 
if CCQ holds, then
SCCQ  holds if and only if 
\begin{equation}\label{eq:CQ support 1d}
\emptyset\ne
\Xi(A,B) 
=\set{Y}{(Y,0)\in \Omega(A,B)+(\{0\}\times \cV)},
%
\end{equation} 
where $\Xi(A,B)$ is defined in \eqref{eq:sccq}.

\item[(b)] (PCQ)  PCQ holds for $p$ if and only if 
\begin{equation}\label{eq:PCQ support}
\pos(\Omega_2(A,B)+\cV)= \lin (\Omega_2(A,B)+\cV),
\end{equation}
where $\Omega_2(A,B)$ is defined in \eqref{eq:C}.

\item[(c)]  (BPCQ) The conditions 
\begin{gather}
\barrier \cV\cap\cK_A\neq \emptyset\AND  \cl(\barrier \cV)\cap \cK_A = \{0\} \label{eq:CQ support 2a},\\
\barrier \cV\cap\cK_A  \; \text{is nonempty and bounded},
\label{eq:CQ support 2b}\\
	\barrier\cV\cap\cK_A\neq \emptyset  \AND  
	\cV^\infty +\cK_A^\circ =\bS^n
	\label{eq:CQ support 2c}
\end{gather}
are each equivalent to BPCQ for $p$, hence imply \eqref{eq:PCQ support}.
\end{itemize}

\end{lemma}

\begin{proof}   Observe that with $h=\sigma_\cV$ we have $\dom h=\barrier \cV$ and  $\hzn h^*=\cV^\infty$.
\smallskip

\noindent
(a)  \eqref{eq:CQ support 1a} is  condition (i) in Lemma \ref{lem:CCQ equiv} for $h=\sigma_\cV$, while \eqref{eq:CQ support 1b} is condition  (iii).  
Employing the results in \cite[Section 3.3, Exercise 16]{BoL 00}) we have  that
\eqref{eq:CQ support 1b} holds if and only if $\cl(\barr \cV-\cK_A)=\bS^n$.
The final statement follows from \eqref{eq:sccq} in the definition of SCCQ.

\noindent
(b) This is an application of \eqref{eq:RintChar} and the definition of PCQ.
\smallskip

\noindent
(c) As the horizon cone of any cone is its closure, we  see that \eqref{eq:CQ support 2a} is exactly BPCQ  (for $h=\sigma_\cV$), while the equivalence to \eqref{eq:CQ support 2b} follows from Lemma \ref{lem:PCQ Equiv} (b). The equivalence of \eqref{eq:CQ support 2c} to the former follows from the fact that 
\eqref{eq:CQ support 2a} holds if and only if
\(\cl(\cV^\infty+\cK_A^\circ)=\bS^n,
\)
see \cite[Section 3.3, Exercise 16]{BoL 00}), where the closure can be dropped by interpreting \cite[Theorem 6.3]{RTR70} accordingly.
\end{proof}

\noindent
The additivity of support functions tells us that
\begin{equation}
p(X)=\inf_{V\in\bS^n} \sigma_{\Sigma}(X,V)\quad (X\in \R^{n\times m}),
\end{equation}
where 
\begin{equation}\label{eq:Sigma}
\Sigma:=  
\Omega(A,B)+\{0\}\times \cV\subset \bE.
\end{equation}
In particular, this implies that $p(\lam X)=\lam p(X)$, for all $\lam >0$ and
$p(X_1+X_2)\le p(X_1)+p(X_2)$. Hence, if $p$ is proper, it
is a support function.
In addition, by \eqref{eq:sccq},
\(
\Xi(A,B)=\set{Y}{(Y,0)\in \Sigma}
\)
is the set featured in \eqref{eq:sccq}, \eqref{eq:Xi}, and 
\eqref{eq:CQ support 1d}.

\begin{proposition}\label{prop:h support simple} Let $p$ be given by  \eqref{eq:p Support Func Case}. Then the following hold: 
\begin{itemize}
\item[(a)] $p\in \Gamma_0(\Rnm)$ (i.e. $p=p^{**}$)   under  condition \eqref{eq:PCQ support}, and, hence, under any of the conditions 
\eqref{eq:CQ support 2a}-\eqref{eq:CQ support 2c}.  Moreover, this is also true  under any condition in \eqref{eq:CQ support 1a}-\eqref{eq:CQ support 1c} if, in addition,  \eqref{eq:CQ support 1d} 
or \eqref{eq:PCQ support} holds, in which case $p$ is finite-valued. 

\item[(b)] $p^*=\delta_{\cl \Sigma}(\cdot,0)$ where the closure is superfluous (i.e. $\Sigma$ is closed) under any of the conditions \eqref{eq:CQ support 1a}-\eqref{eq:CQ support 1c}, 
in which case  $p^*=\delta_{\Xi(A,B)}$.

\item[(c)] If any of  \eqref{eq:CQ support 1a}-\eqref{eq:CQ support 1c} hold then $p\equiv -\infty$ or $p=\sigma_{\Xi(A,B)}$ is finite-valued. The latter is the case if and only if \eqref{eq:CQ support 1d} holds, which is valid under  \eqref{eq:PCQ support}.

\end{itemize}

\end{proposition}
\begin{proof} (a) The first statement follows from Lemma \ref{lem: CQ support} and Theorem \ref{th:p under PCQ}. The second uses Lemma \ref{lem: CQ support},  Theorem \ref{th:inf-proj psi} (c)  and Corollary \ref{cor:P+C=S}.
\smallskip

\noindent
(b) By \cite[Exercise 3.12]{RoW98} and \cite[Proposition 10]{BGH17}, $\Sigma$ is closed if $(-\cK_A^\circ)\cap \cV^\infty=\{0\}$, i.e. under any condition in \eqref{eq:CQ support 1a}-\eqref{eq:CQ support 1c}, see Lemma \ref{lem: CQ support} (a). Moreover, $p^*=\sigma_{\Sigma}^*(\cdot,0)=\delta_{\cl \Sigma}(\cdot, 0)$, see \cite[Proposition 11.23 (c)]{RoW98}.
\smallskip

\noindent
(c) Follows from (a),  (b) and Theorem \ref{th:inf-proj psi} c II), as well as Corollary \ref{cor:P+C=S}.
\end{proof}

%
%
%

\subsection{The case $B=0$}
We now consider the case when $B=0$. Recall from 
\cite[Theorem 11]{BGH17} that this implies that 
$\sig_{\Omega(A,0)}$ is a gauge function. Similarly, if $0\in \cV$, then
$\sig_\cV$ is also a gauge, in fact, $\sig_\cV=\gamma_{\cV^\circ}$,  cf.  \cite[ Example 11.19]{RoW98}.

 This combination of assumptions has interesting consequences when the geometries of the sets $\cV$ and 
$-\cK_A^\circ$ are compatible in the following sense. 

\begin{definition}[Cone compatible gauges]
Given a closed, convex cone $K\subset\cE$, we define an ordering
on $\cE$ by  $x\preceq_Ky$ if and only if $y-x\in K$. 
A gauge $\gamma$ on
$\cE$ is said to be compatible with this ordering if 
\\
\centerline{\(
\gamma(x)\le\gamma(y)\ \mbox{ whenever }\  0\preceq_K x\preceq_K y.
\)
}
\end{definition}

\noindent
The following lemma provides a  characterization of cone compatible gauges and provides a very useful tool for determining is a gauge is
compatible with a given cone.

\begin{lemma}[Cones and compatible gauges]\label{lem:cones and gauges}
Let $0\in C\subset \cE$ be a  closed, convex set, and
let $K\subset\cE$ be a closed, convex cone. 
Then $\gamma_C$ is compatible with the ordering $\preceq_K$
if and only if  
\( 
K\cap(y-K)\subset C\quad (y\in K\cap C). 
\) 
\end{lemma}

\begin{proof}
Note that, for $y\in K$, we have
\(
K\cap(y-K)=\set{x}{0\preceq_K x\preceq_K y}.
\)
Suppose that $\gamma_C$ is compatible with $K$, and let
$y\in C\cap K$. If $x\in K\cap(y-K)$, then
$\gamma_C(x)\le \gamma_C(y)\le 1$, and, consequently,  $K\cap(y-K)\subset C$.

Next suppose $K\cap(y-K)\subset C$ for all $y\in K\cap C$, 
and let $x,y\in \cE$
be such that $0\preceq_K x\preceq_K y$. Then, $y\in K$ and
$x\in K\cap(y-K)$.
We need to show that $\gamma_C(x)\le \gamma_C(y)$.
If $\gamma_C(y)=+\infty$, this is trivially the case, so we may as well
assume that $\gamma_C(y)=:\bt<+\infty$. If $\bt>0$, then $\bt^{-1}y\in C\cap K$ and
$\bt^{-1}x\in K\cap(\bt^{-1}y-K)\subset C$. Hence,
$\gamma_C(\bt^{-1}y)=1\ge \gamma_C(\bt^{-1}x)$, and so, 
$\gamma_C(x)\le \gamma_C(y)$ as desired. In turn, if $\bt=0$,
then $ty\in K\cap C\;(t>0)$, so that $tx\in K\cap(ty-K)\subset C\;(t>0)$,
i.e., $x\in C^\infty$ and so $\gamma_C(x)=0$.
\end{proof}

\begin{corollary}[Infimal projection with a gauge function]\label{cor:h is gauge} Let $p$ be given by \eqref{eq:p Support Func Case} where 
$\cV$ is a nonempty, closed, convex subset of $\Sn$. Suppose that $B=0$. Under  any of the conditions \eqref{eq:CQ support 1a}-\eqref{eq:CQ support 1c} we have:

\noindent
(a)
\(
p^*= \delta_{\Xi(A,0)},
\)
where
\(
\Xi(A,0)=\set{Y}{AY=0,\;\exists W\in \cV:\; AW=0,\;\half YY^T\preceq W}.
\)

\noindent
(b)
If $0\in\cV$ and $\gamma_{\cV}$ is compatible with the ordering induced by  $-\cK_A^\circ$, 
then 
\begin{equation}\label{eq:p conj for h gauge and B=0}
\begin{aligned}
p^*(Y)
=\indicator{Y}{\set{Y}{AY=0,\;\gauge{\half YY^T}{\cV}\le 1}}
=\indicator {\half YY^T}{ (-\cK_A^\circ)\cap \cV}.
\end{aligned}
\end{equation}
\end{corollary}
\begin{proof} (a) This follows from  Proposition  \ref{prop:h support simple}, 
\eqref{eq:CQ support 1d} with $B=0$, and using the representation of $\cK_A$ in Proposition \ref{prop:K_A}.
\smallskip

\noindent
(b) First observe that $ -\cK_A^\circ= \set{W\in\Snp}{\rge W\subset \ker A}$, see  Proposition  \ref{prop:K_A} (b), recall that $\rge Y= \rge YY^T\;(Y\in \R^{n\times m})$  and, since $0\in\cV$, $V\in \cV$ if and only if $\gamma_\cV(V)\leq 1$. Exploiting these facts and the compatibility hypothesis, we see that
\begin{eqnarray*}
Y\in \Xi(A,0)
&\Longleftrightarrow \lefteqn{AY=0, \;\exists W\in \cV: AW=0, \half YY^T\preceq W}\\ 
& \Longrightarrow & AY=0, \;\exists W\in\cV: \; \gamma_\cV(W)\geq \gamma_\cV\left(\frac{1}{2}YY^T\right)\\
& \Longleftrightarrow & AY=0,\;\gamma_\cV\left(\frac{1}{2}YY^T\right)\leq 1\\
& \Longleftrightarrow &  AY=0, \;\frac{1}{2}YY^T\in \cV\\
& \Longleftrightarrow &  \rge YY^T \subset \ker A, \;\frac{1}{2}YY^T\in \cV\\
& \Longleftrightarrow & \frac{1}{2} YY^T\in (-\cK_A^\circ)\cap \cV.
\end{eqnarray*}
Conversely,  we have
\(
\frac{1}{2} YY^T\in (-\cK_A^\circ)\cap \cV
  \Longleftrightarrow  AY=0, \ Y\in \cK_A, \mbox{ and } \frac{1}{2} YY^T\in \cV.
\)
Taking $W=\frac{1}{2} YY^T$, we see that $Y\in \Xi(A,0)$.
Therefore  (b) follows from (a). 
\end{proof}

\noindent
When the support function $h$ is taken to be a linear functional, 
we obtain the following remarkable result.
Here $\|\cdot\|_*$ denotes the nuclear norm\footnote{For a matrix $T$ the nuclear norm  $\|T\|_*$ is the sum of its singular values.}.

\begin{corollary}[$h$ linear]\label{cor:h linear} Let $p:\R^{n\times m}\to \rbar$ be defined by
\[
p(X)=\inf_{V\in\bS^n} 
\varphi(X,V)+\ip{\bar U}{V}
\]
for some $\bar U\in\bS^n_{+}\cap \Ker_n A$ and set
$
C(\bar U):=\set{Y\in\Rnm}{\frac{1}{2}YY^T\preceq \bar U}.
$
 Then:
\begin{itemize}
\item[(a)]  $p^*=\delta_{C(\bar U)}$ is closed, proper, convex. 

\item[(b)]  $p=\sigma_{C(\bar U)}=
\gamma_{C(\bar U)^\circ}$  is sublinear, finite-valued,  nonnegative and symmetric (i.e.~a seminorm).

\item[(c)] If $\bar U\succ 0$   with $2\bar U=LL^T$ ($L\in\R^{n\times n}$) and  $A=0$ then 
\(
p=\sigma_{C(\bar U)}=\|L^T(\cdot)\|_*,
\) 
i.e. $p$ is a norm with $C(\bar U)^\circ$ as its unit ball and $\gamma_{C(\bar U)}$ as its dual norm.

\item[(d)] If $\bar U\succ 0$,  then $C(\bar U)$ and 
$C(\bar U)^{\circ}$ are compact, convex, 
symmetric\footnote{We say the set $S\subset\cE$ symmetric if $S=-S$.} with $0$ in their interior, thus $\pos C(\bar U)=\pos C(\bar U)^\circ=\bS^n.$

 
\end{itemize}

\end{corollary}

\begin{proof} (a)  Observe that $h:=\ip{\bar U}{\cdot}=\sigma_{\{\bar U\}}$. Hence the machinery from above  applies with $\cV=\{\bar U\}$. As $\cV$ is bounded, CCQ   is trivially satisfied 
(cf.~\eqref{eq:CQ support 1a}-\eqref{eq:CQ support 1c}).
Note that $0\in C(\bar U)\ne \emptyset$. Given $Y\in C(\bar U)$,
we must have $\rge Y\subset \ker A$ since otherwise there is a nonzero
$z\in (\ker A)^\perp$ with $Y^Tz\ne 0$ yielding $
0< \tnorm{Y^Tz}^2\le 2 z^T\bar Uz=0.
$
Consequently,
$
C(\bar U)=\set{Y\in\Rnm}{AY=0,\ \half YY^T-\bar U\in\cK_A^\circ}
=\Xi(A,0)\ne\emptyset,
$
and the result follows from Proposition \ref{prop:h support simple} (b).
\smallskip

\noindent
(b) This follows from \cite[Theorem 14.5]{RTR70}, part (a), and the fact that $0\in C(\bar U)$.
\smallskip


\noindent
(c) Consider the case $\bar U=\frac{1}{2}I$: By part (a),
we have  $p^*=\delta_{\set{Y}{YY^T\preceq I}}$.   Observe that
$\set{Y}{YY^T\preceq I}=\set{Y}{\|Y\|_2\leq 1}=:\bB_{\Lambda}
$ 
is the closed unit ball of the spectral norm. Therefore, $p=\sigma_{\bB_{\Lambda}}=\|\cdot\|_{\bB_{\Lambda}^\circ}=\|\cdot\|_*$.

To prove the general case suppose that  $2\bar U=LL^T$. Then it is clear that $C(\bar U)=\set{Y}{L^{-1}Y\in C(\frac{1}{2}I)}$, 
and therefore 
\begin{eqnarray*} 
p(X) & = & \sigma_{C(\bar U)} (X)\\
& = & \sup_{Y:L^{-1}Y\in C(\frac{1}{2} I)}\ip{Y}{X}\\
& = & \sup_{L^{-1}Y\in C(\frac{1}{2}I)}\ip{L^{-1}Y}{L^TX}\\
& = & \sigma_{C(\frac{1}{2}I)}(L^TX)\\
& = & \|L^TX\|_*.
\end{eqnarray*}
Here the first identity is due to part (b) (with $A=0$) and the last one follows from the special case considered at the start of the proof.


\item[(d)] Follows from (c) using \cite[Theorem 15.2]{RTR70}.
\end{proof}

\noindent
We point out that  Corollary \ref{cor:h linear}  generalizes the nuclear norm smoothing result by Hsieh and Olsen \cite[Lemma 1]{HsO 14} and complements  \cite[Theorem 5.7]{BuH15}

%
%
%
%

\section{$h$ is an indicator function}\label{sec:Indicator}
\noindent
We now suppose that the function $h$ in \eqref{eq:psi} is the indicator
$h:=\del_\cV$ for some nonempty, closed, and  convex set $\cV\in\bS^n$:
\begin{equation}\label{eq:p indicator}
p(X) = \displaystyle \inf_{V\in\bS^n} 
\varphi(X,V)+ \delta_\mathcal{V}(V).
\end{equation}

\noindent
We begin by interpreting the constraint qualifications from Section \ref{sec:CQ}. 
Here, and for the remainder of this section, 
$h=\delta_{\cV}$ and so $\dom h=\cV$ and $\dom h^*=\barr \cV$.

\begin{lemma}[Constraint qualifications  for \eqref{eq:p indicator}]\label{lem: CQ indicator} Let $p$ be given by \eqref{eq:p indicator}.  Then the following hold:
\begin{itemize}
\item [(a)] (CCQ) The  conditions 
\begin{eqnarray}
 \cV\cap \inter \cK_A\neq  \emptyset, \label{eq:CQ indicator 1a}\\
\overline\cone \cV-\cK_A  = \bS^n \label{eq:CQ indicator 1b}
\end{eqnarray}
are each equivalent to CCQ for $p$.
Moreover, if CCQ holds, then SCCQ holds if and only if 
\begin{equation}\label{eq:indicator sccq}
\emptyset\ne \Xi(A,B)
=\set{Y\in\Rnm}{(Y,0)\in\Omega(A,B)+(\{0\}\times \barr\cV)}.
\end{equation}
\item[(b)] (PCQ) The PCQ holds for $p$ if and only if 
\begin{equation}\label{eq:PCQ indicator}
\pos (\Omega_2(A,B))+\barr\cV= \lin (\Omega_2(A,B)+\barr \cV).
\end{equation}

\item[(c)] (BPCQ) The conditions 
\begin{eqnarray}
\cV\cap\cK_A\neq \emptyset \AND \mathcal{V}^\infty \cap \cK_A=\{0\} \label{eq:CQ indicator 2a},\\
\cV\cap\cK_A\neq \emptyset \; \text{is bounded},\label{eq:CQ indicator 2b}\\
	\cV\cap\cK_A\neq \emptyset \AND  \barrier \cV+\cK_A^\circ =\bS^n\label{eq:CQ indicator 2c}
\end{eqnarray}
are each equivalent to BPCQ for $p$, hence imply \eqref{eq:PCQ indicator}.

\end{itemize}

\end{lemma}
\begin{proof}
(a) First, observe that , with $h=\del_\cV$,  condition (i) in Lemma \ref{lem:CCQ equiv}  is exactly \eqref{eq:CQ indicator 1a}. By the same lemma this is equivalent to 
\(
\hzn \sigma_\cV\cap (-\cK_A^\circ)=\{0\}.
\)
Moreover, since $\sigma_\cV=\sigma_\cV^\infty$, we have 
\(
\hzn \sigma_\cV=\set{V}{\sigma_{\cV}(V)\leq 0}=(\cone \cV)^\circ.
\)
Invoking the results in \cite[Section 3.3, Exercise 16 (a)]{BoL 00} implies that 
\(
\hzn \sigma_\cV\cap (-\cK^\circ_A)=\{0\}
\)
if and only if \(\cl(\overline\cone \cV-\cK_A)  = \bS^n,
\)
where the closure in the latter statement can clearly be dropped, e.g. by interpreting  \cite[Theorem 6.3]{RTR70} accordingly.
\smallskip

\noindent
(b) Use \eqref{eq:RintChar} to infer that PCQ holds for $p$ if and only if 
\[
\pos (\Omega_2(A,B))+\barr V=\pos(\Omega_2(A,B)+\barr V)=\lin(\Omega_2(A,B)+\barr \cV).
\]
 
\smallskip

\noindent
(c) The equivalences of BPCQ, \eqref{eq:CQ indicator 2a}, and \eqref{eq:CQ indicator 2b} are clear.   Since $\cV^\infty$ and $\cl(\barrier \cV)$ are paired in polarity, see \eqref{eq:BarrierPolarity}, \cite[Section 3.3, Exercise 16 (a)]{BoL 00} implies that
\(
\cV^{\infty} \cap \cK_A =\{0\}\)
if and only if  
\(\cl(\barrier \cV+\cK_A^\circ) = \bS^n,
\)
where the closure in the latter statement can be dropped as in (a). This establishes all equivalences. 
\end{proof}

\noindent
The following result provides sufficient conditions for  $p$ being closed, proper, convex when
$h$ is an indicator function.

\begin{corollary}\label{cor:p=p^**} Let $p$ be given by \eqref{eq:p indicator}. Then 
 $p\in \Gamma_0(\R^{n\times m})$ under any of the following conditions: 
 \begin{itemize}
 \item[(i)] \eqref{eq:indicator sccq} holds along with
 either \eqref{eq:CQ indicator 1a} or \eqref{eq:CQ indicator 1b}.
 \item[(ii)] \eqref{eq:PCQ indicator} holds.
 \item[(iii)] Any one of  
 \eqref{eq:CQ indicator 2a}-\eqref{eq:CQ indicator 2c} holds.
 \end{itemize}
\end{corollary}
\begin{proof} Follows   from Lemma \ref{lem: CQ indicator} and  Theorem \ref{th:inf-proj psi} (c) and Theorem \ref{th:p under PCQ}, respectively. 
\end{proof}

\noindent
The case $A=0$ and $B=0$ is of particular interest in applications to
variational Gram functions in Section \ref{sec:VGF}.

\begin{corollary}\label{cor:Indicator PCQ} Let $p$ be given as in 
\eqref{eq:p indicator} with $A=0$ and $B=0$ so that
$\cK_A=\Sn_+$ and $\cK_A^\circ=\Sn_-$. Assume that $\cV\cap \bS^n_{+}\ne\emptyset$.
 Then 
\\
\centerline{$\text{PCQ}\iff \text{SPCQ} \iff \Sn_{-}+\barr \cV=\Sn\iff \text{BPCQ}$.}
\\
Moreover, $p\in \Gamma_0(\R^{n\times m})$ under any of  following conditions:
\begin{itemize}
\item[(i)] (SCCQ) 
$\set{Y\in\Rnm}{\exists\ T\in\barr \cV\, :\, \half YY^T\preceq T}\ne
\emptyset$
and $\cV\cap \bS^n_{++}\neq \emptyset$; 
\item[(ii)] (PCQ)  $\Sn_{-}+\barr \cV=\Sn$;
\item[(iii)] ((B/S)PCQ) $\emptyset\ne \cV\cap \Sn_{+}$ is bounded.
	\end{itemize}
\end{corollary}

\begin{proof} 
First note that $\Xi(0,0)=\set{Y\in\Rnm}{\exists\ T\in\barr \cV\, :\, \half YY^T\preceq T}$
and $\Omega_2(0,0)=\bS_{-}^n=\cK^\circ_{A}$. 
The first statement now follows from 
Lemma \ref{lem: CQ indicator} and the definition of PCQ  and SPCQ, resepectively, since the span of a set with interior is the whole space. The remaining implications follow from
 Corollary \ref{cor:p=p^**} and Lemma \ref{lem: CQ indicator}.
\end{proof}

\noindent
We directly compute the conjugate $p^*$ using techniques from \cite[Theorem 3.2]{BuH15}.
%
%

\begin{theorem}[Infimal projection with an indicator function]\label{th:indicatorConjugateGeneral}
Let $p$ be given by \eqref{eq:p indicator}. 
Assume that 
\begin{equation}\label{eq:nonempty domain 2}
\emptyset\ne\dom(\varphi+\delta_{\cV})=\set{(X,V)\in\bE}{V\in \cV\cap\cK_A\AND \rge \binom{X}{B}\subset \rge M(V)}.
\end{equation}
Then $p^*:\R^{n\times m}\to\rbar$ is given by
\\
\centerline{
$
p^*(Y) = \half
\support{YY^T}{\cV\cap\cK_A}+\indicator{Y}{\set{Z}{AZ=B}}.
$}
\\
In particular, for $A=0$ and $B=0$ we obtain 
\(
p^*(Y)=\frac{1}{2}\sigma_{\cV\cap \bS^n_+}\left(YY^T\right).
\)
\end{theorem}
\begin{proof}
By \eqref{eq:GMF} and our assumption that $\emptyset\ne\dom(\varphi+\delta_{\cV})$, we have
\[
\begin{aligned}
p^*(Y) &= \sup_X\left[ \langle X, Y \rangle - \inf_V 
\varphi(X,V) + \delta_\mathcal{V}(V)\right]\\
&= \sup_V\sup_X \left[\langle X, Y \rangle - \sigma_{ \Omega(A,B)} (X,V) - \delta_\mathcal{V}(V)\right]\\
&= \sup_{V\in\cV\cap\cK_A}
\sup_{\rge \binom{X}{B}\subset \rge M(V)} 
\ip{X}{Y}-\frac{1}{2}\tr\left(\binom{X}{B}^TM(V)^\dagger \binom{X}{B}\right) ,
\end{aligned}
\]
for $Y\in\R^{n\times m}$.
Since $\rge \binom{X}{B}\subset \rge M(V)$, we make the substitution 
$
M(V)\binom{U}{W} = \binom{X}{B}
$
to obtain
\begin{eqnarray*}
p^*(Y) &=& \sup_{V\in\cV\cap \cK_A}
\sup_{\substack{U, W\\ AU = B}} 
\tr\left( -\frac{1}{2}\binom{U}{W}^TM(V)\binom{U}{W} + Y^T(VU + A^TW) \right)
\\
&=& \sup_{V\in\cV\cap \cK_A}-
\sum_{i=1}^m\inf_{\substack{u_i, w_i\\ Au_i = b_i}}\left( \frac{1}{2}\binom{u_i}{w_i}^TM(V)\binom{u_i}{w_i} - y_i^TVu_i - w_i^TAy_i \right)
\\
&=& \sup_{V\in\cV\cap \cK_A}-
\sum_{i=1}^m\inf_{\substack{u_i, w_i\\ Au_i = b_i}}
\left( \frac{1}{2}u_i^TVu_i - \langle Vy_i, u_i \rangle + \langle w_i, b_i - Ay_i \rangle \right)
\\
&=& \sup_{V\in\cV\cap \cK_A}-
\sum_{i=1}^m\left[
\inf_{Au_i = b_i}
\left( \frac{1}{2}u_i^TVu_i - \langle Vy_i, u_i \rangle\right) + 
\inf_{w_i}\left(\langle w_i, b_i - Ay_i \rangle \right)\right]
\\
&=& \delta_{\set{Z}{AZ=B}}(Y)+ \sup_{V\in\cV\cap \cK_A}-
\sum_{i=1}^m
\inf_{Au_i = b_i}
\left( \frac{1}{2}u_i^TVu_i - \langle Vy_i, u_i \rangle\right),
\end{eqnarray*}
where the final equality follows since 
$\delta_{\set{y}{Ay=b_i}}(y_i)=\sup_{w_i}\ip{w_i}{Ay_i-b_i } \;(i=1,\dots,m)$.
By hypothesis $\rge B\subset\rge A$, and so, by \cite[Theorem 3.2]{BuH15}\[
-\half \binom{Vy_i}{b_i}^TM(V)^\dagger\binom{Vy_i}{b_i}=
\inf_{Au_i = b_i}
\left( \frac{1}{2}u_i^TVu_i - \langle Vy_i, u_i \rangle\right)
\quad (i=1,\dots,m),
\]
Therefore, when $AY=B$, we have
\begin{eqnarray*}
p^*(Y) 
&=& \sup_{V\in\cV\cap \cK_A}-\sum_{i=1}^m-\frac{1}{2}\binom{Vy_i}{b_i}^TM(V)^\dagger\binom{Vy_i}{b_i}\qquad \qquad\qquad
\left(\begin{matrix}\text{where $Ay_i=b_i$ so}\\ 
\binom{Vy_i}{b_i}=M(V)\binom{y_i}{0}\end{matrix}\right)
\\
&=& \sup_{V\in\cV\cap \cK_A}\frac{1}{2}
\sum_{i=1}^m\left(M(V)\binom{y_i}{0}\right)^TM(V)^\dagger\left(M(V)\binom{y_i}{0}\right)\\
&=& \sup_{V\in\cV\cap \cK_A}
\frac{1}{2}\sum_{i=1}^m \binom{y_i}{0}^TM(V)\binom{y_i}{0}^T\\
&=& \sup_{V\in\cV\cap \cK_A}\frac{1}{2}
\sum_{i=1}^my_i^TVy_i\\
&=& \sup_{V\in\cV\cap \cK_A}\frac{1}{2}\tr (Y^TVY),
\end{eqnarray*}
which proves the general expression for $p^*$. The case $A=0, B=0$ follows.
\end{proof}

%
%

\begin{corollary}\label{cor:Subdiff p indicator}  Let $p$ be given by \eqref{eq:p indicator}.  If SCCQ holds, i.e.,
\\
\centerline{$\cV\cap \inter \cK_A\neq \emptyset$ and
$ \set{Y\in\Rnm}{(Y,0)\in\Omega(A,B)+(\{0\}+\barr\cV)}\ne \emptyset$,}
\\ 
then 
\[
\p p(\bar X)=\argmax_{Y}\{\ip{\bar X}{Y}-\inf_{(Y,T)\in\Omega(A,B)}\sigma_\cV(-T)\}
\]
is nonempty and compact for all $\bar X\in \R^{n\times m}$. 
Alternatively, if $\cV\cap \inter \cK_A\neq \emptyset$ (CCQ) and 
\\
\centerline{$\pos \Omega_2(A,B)+\barrier \cV=\lin(\Omega_2(A,B)+\barrier \cV)$ (PCQ)} 
\\
hold,
then 
\[
\p p(\bar X)=\set{\bar Y}{\exists \bar V,\bar T:\;-\bar T\in N_{\cV}(\bar V),\; (\bar Y,\bar T)\in \p \vphi(\bar X,\bar V)}
\]
is nonempty and compact for all $\bar X\in \R^{n\times m}$.
 \end{corollary}

\begin{proof} This follows from  Proposition \ref{prop:SD + PS}  in combination with  Lemma \ref{lem: CQ indicator}.
\end{proof}

\subsection{$B=0$ and $0\in\cV$}

We now consider the important special case of $p$ given by \eqref{eq:p indicator}  where $0\in\cV$ and $B=0$. In this case $p$ turns out to be a squared  gauge function, see Corollary \ref{cor:Gauge case II}.   We start with a technical lemma.

%
%
%
%
%
\begin{lemma}\label{lem:polar sum}
Let $C,K\subset\bE$ be nonempty, convex 
with $K$ being a cone. Then
$(C+K)^\circ=C^\circ\cap K^\circ$. If
$C+K$ is closed with $0\in C$, then $(C^\circ\cap K^\circ)^\circ=C+K$.
In particular, the set $C+K$ is closed if $C$ and $K$ are closed and 
$K\cap(-C^\infty)=\{0\}$.
\end{lemma}
\begin{proof}
Clearly, $C^\circ\cap K^\circ\subset (C+K)^\circ$. Conversely, if
$z\in (C+K)^\circ$, then $\ip{z}{x+ty}\le 1$ for all $x\in C$, $y\in K$, and $t>0$.
Multiplying this inequality by $t^{-1}$ and letting $t\to\infty$, we see that $z\in K^\circ$.
By letting $t\downarrow 0$, we see that $z\in C^\circ$. 

Now assume that $C+K$ is closed with $0\in C$.  Then 
 $C+K$ is closed and convex with $0\in C+K$.  
Hence, by \cite[Theorem 14.5]{RTR70},
$C+K=(C+K)^{\circ\circ}=(C^\circ\cap K^\circ)^\circ$. 

The final statement of the
lemma follows from
\cite[Corollary 9.1.1]{RTR70}.
\end{proof}

\noindent
The first result in this section is concerned with a representation of the conjugate $p^*$ under the standing assumptions.

\begin{corollary}[The gauge case I]\label{cor:Gauge case I}
Let $p$ be given by  \eqref{eq:p indicator}  with $0\in\cV$ and $B=0$ and let $P$ be  the orthogonal projection onto $\ker A$. 
Moreover, let 
 \[
\cS:= 
\set{W\in\Sn}{\rge W\subset\ker A}
=\set{W\in\Sn}{W=PWP}.
\] 
Assume that
\\
\centerline{$\emptyset\ne\set{(X,V)\in\bE}{V\in \cV\cap\cK_A\AND \rge \binom{X}{0}\subset \rge M(V)}.$}
\\
Then the following hold:
\begin{itemize}
\item[(a)] We have 
\[
p^*(Y) =\frac{1}{2}\support{YY^T}{(\cV\cap\cK_A)+\cS^\perp}=\half\gauge{YY^T}{(\cV\cap\cK_A)^\circ \cap\cS}
\]
where  $\cS^\perp=\set{V\in\Sn}{PVP=0}$.  In particular, $p^*$ is positively homogeneous of degree 2.

\item[(b)] If $\cV^\circ+\cK_A^\circ$ is closed (e.g. when $\cK_A^\circ\cap -(\cone{\cV})^\circ=\{0\}$) then
\begin{equation}\label{eq:gauge dist}
p^*(Y)  =
\half\gauge{YY^T}{(\cV^\circ\cap\cS) + \cK_A^\circ},
\end{equation}
where $\dom p^*=\set{Y}{YY^T\in \cone{(\cV^\circ\cap\cS)} +\cK_A^\circ}$.
\end{itemize}
\end{corollary}
\begin{proof}
(a) By Theorem \ref{th:indicatorConjugateGeneral}, we have
\[
\begin{aligned}
p^*(Y) =& 
\half\support{YY^T}{\cV\cap\cK_A}+\delta_{\set{Z}{AZ=0}}(Y)
\\ =&
\half\support{YY^T}{\cV\cap\cK_A}+\frac{1}{2}\indicator{YY^T}{\cS}
\\ =&
\half\support{YY^T}{\cV\cap\cK_A}+\half\support{YY^T}{\cS^\perp}
\\ =&
\half\support{YY^T}{(\cV\cap\cK_A)+\cS^\perp}
\\ =&
\half\gauge{YY^T}{(\cV\cap\cK_A)^\circ \cap\cS}.
\end{aligned}
\]
Here the first equality uses Theorem \ref{th:indicatorConjugateGeneral},  the second equality follows from the fact that
$\rge Y=\rge YY^T$, the third can be seen from \cite[Example 7.4]{RoW98}, and the final equivalence follows from \cite[Theorem 14.5]{RTR70} and Lemma \ref{lem:polar sum}.

\smallskip

\noindent
(b) If $\cV^\circ+\cK_A^\circ$ is closed, then Lemma \ref{lem:polar sum}
also tells us that $(\cV\cap\cK_A)^\circ=\cV^\circ+\cK_A^\circ$. Since $\cK_A^\circ\subset \cS$, see Lemma \ref{prop:K_A} (b), we have
\(
(\cV^\circ+\cK_A^\circ)\cap\cS=(\cV^\circ\cap\cS)+\cK_A^\circ
\)
which, using (a), gives the first equivalence in \eqref{eq:gauge dist}. 
\end{proof}

\noindent
Our final goal is to show that  $p$, under the standing assumption in this section, is a squared gauge. 
%
%
Here we denote by $\bB_F$ the (closed) unit ball in the Frobenius norm. 

\begin{corollary}[The gauge case II]\label{cor:Gauge case II}
Let $p$ be as in Theorem \ref{th:indicatorConjugateGeneral} with $0\in\cV$ and $B=0$,
and assume that \eqref{eq:nonempty domain 2} holds.
Let $P\in \R^{n\times n}$ be the orthogonal projector on $\ker A$ and define the (closed, convex) sets
\[
\hcVp:=\set{L\in \R^{n\times n}}{LL^T\in P(\cV\cap\cK_A)P},
\quad
\cF:=\set{LZ}{L\in \hcVp,\; Z\in \bB_F},
\]
and the subspace 
$
\cU:=\Ker_mA.\footnote{Hence  $\cU^\perp=\Rge_mA^T$.}
$
Then 
\[
p=\frac{1}{2}\gamma^2_{\cF+\cU^\perp}\AND p^*=\frac{1}{2}\gamma^2_{\cF^\circ\cap\cU}.
\]
In particular, for $A=0$ and $\cF:=\set{LZ}{LL^T\in \cV\cap \bS^n_+,\; Z\in\bB_F}$ we obtain
\[
p=\half \gamma^2_\cF\AND p^*=\gamma^2_{\cF^\circ}.
\]
\end{corollary}
\begin{proof} For all $Y\in \R^{n\times m}$, by Theorem \ref{th:indicatorConjugateGeneral}  and the definition of $\cU$, we have
\[
p^*(Y) =\frac{1}{2}\sigma_{\cV\cap\cK_A}(YY^T)+\delta_\cU(Y)= \half \sup_{V\in\cV\cap \cK_A}\ip{PVP}{YY^T}+\delta_\cU(Y).
\]
In turn, by the definitions of $\hcVp$ and the  Frobenius norm, the latter equals
\[
 \half \sup_{L\in \hcVp}\ip{LL^T}{YY^T}+\delta_\cU(Y)=\half \sup_{L\in\hcVp}\|L^TY\|_F^2+\delta_\cU(Y).
 \]
On the other hand, by the monotonicity and continuity of $t\in \R_+\mapsto t^2$ as well as the  self-duality of the Frobenius norm, we find that the second term can be written as
\[
 \half \left[\sup_{L\in\hcVp}\|L^TY\|_F\right]^2+\delta_\cU(Y)= \half \left[\sup_{(Z,L)\in \bB_F\times\hcVp}\ip{L^TY}{Z}\right]^2+\delta_\cU(Y).
\]
Using the definition of $\cF$ and the convention  $(+\infty)^2=+\infty$, we can rewrite 
this equivalence as 
\( 
\half\sigma_{\cF}(Y)^2+\delta_{\cU}(Y)= \half\left[\sigma_{\cF}(Y)+\delta_{\cU}(Y)\right]^2.
\)
All in all, using the latter,   \cite[Example 11.4]{RoW98}, and \cite[Example 11.19]{RoW98} and the  polar cone calculus from, e.g., \cite[p.~70]{BoL 00}, we conclude that
\[
p^*(Y) \!=\! \half\left[\sigma_{\cF}(Y)+\delta_{\cU}(Y)\right]^2 \!=\! \half\left[\sigma_{\cF}(Y)+\sigma_{\cU^\perp}(Y)\right]^2 \!=\! \half \sigma^2_{\cF+\cU^\perp}(Y) \!=\! \half\gamma^2_{\cF^\circ\cap\cU}(Y).
\]
This gives the representation for  $p^*$; the one for $p$ follows from 
\cite[Corollary 15.3.1]{RTR70}.
\end{proof}


\subsection{Variational Gram Functions} \label{sec:VGF}

Given a  closed,  convex set  $\cV\subset \bS^n$ define 
\begin{equation}\label{eq:VGF}
\Phi_{\cV}:\R^{n\times m}\to \rbar, \quad 
\Phi_\cV(Y):=\frac{1}{2}\sigma_{\cV\cap \bS^n_+}(YY^T).
\end{equation}
These functions are called {\em variational Gram functions (VGF)} and
were introduced by Jalali, Fazel and Xiao \cite{JFX17}.
They have received attention in the machine learning community due to their orthogonality promoting properties when used as penalty functions, cf. \cite{JFX17}.  

Note that the definition \eqref{eq:VGF} explicitly intersects $\cV$ 
with the positive semidefinite cone $\bS^n_+$ 
while Jalali, Fazel and Xiao 
\cite{JFX17} employ the standing assumption that 
$\Phi_\cV=\Phi_{\cV\cap \bS^n_+}$. These (equivalent) conventions guarantee that $\Phi_\cV$ is convex.  We also scale by $\frac{1}{2}$ since $\Phi_{\cV}$
is positively homogeneous of degree 2.

As an immediate consequence of Theorem \ref{th:indicatorConjugateGeneral}, 
$\Phi_\cV=p^*$ where  
$p$ is defined in \eqref{eq:p indicator} 
with $A=0$, $B=0$ and $\cV\cap\Sn_+\ne\emptyset$.
In addition, the constraint qualifications dramatically simplify in this case.
We have already seen in Corollary \ref{cor:Indicator PCQ} that PCQ, SPCQ and BPCQ are
all equivalent for VGFs.  
We now observe that CCQ and SCCQ are also equivalent.

\begin{lemma}[CCQ$=$SCCQ for VGFs]\label{lem:ccq=sccq}
Let $\Phi_{\cV}$ be given by \eqref{eq:VGF} with $\cV\subset \bS^n$.
Then the condition $\cV\cap \bS^n_+\ne\emptyset$ is equivalent to
\eqref{eq:nonempty domain 2}, and  
\eqref{eq:indicator sccq} is satisfied with $\Phi_{\cV}=p^*$
where $A=0$, $B=0$ and $p$ defined in \eqref{eq:p indicator}.
In particular, CCQ and SCCQ are equivalent where CCQ is given by
$\cV\cap \bS^n_{++}\ne\emptyset$.
\end{lemma} 
\begin{proof}
First note that 
\(
0\in \Xi(0,0)=\set{Y}{\exists\, W\in \barr \cV\, :
\, \half YY^T\preceq W}
\)
since 
$0\in \barr \cV$.
The relationship between $\Phi_\cV$ and $p$ is given in
Theorem \ref{th:indicatorConjugateGeneral}.
\end{proof}

Lemma \ref{lem:ccq=sccq} and the results of the previous section allow us to refine 
\cite[Proposition 4]{JFX17}. 

\begin{proposition}[Conjugate of VGFs and VGFs as Squared Gauges]\label{prop:VGF conjugate} Let $\Phi_\cV$ be given by \eqref{eq:VGF}. 
Under either of the assumptions
		\begin{itemize}
		\item[(i)]  (CCQ) $\cV\cap \bS^n_{++}\neq \emptyset$,
		\item[(ii)] (PCQ) $\cV\cap \bS^n_{+}\neq \emptyset$ is bounded 
		(or equivalently $\bS_{-}^n+\barr \cV=\bS^n$),
	\end{itemize}
we have 
\[
\Phi_{\cV}^*(X)= \inf_{V} \sigma_{\Omega(0,0)}(X,V)+\delta_\cV(V)= \frac{1}{2}\inf_{\begin{smallmatrix}V\in \cV\cap \bS^n_+:\\ \rge X\subset \rge V\end{smallmatrix}} \tr\left(X^TV^\dagger X\right)\quad (X\in \R^{n\times m}).
\]
Under (i), $\Phi_{\cV}^*$ is finite-valued, and under (ii), $\Phi_{\cV}$ is finite-valued.  In addition, if $0\in\cV$ we also have 
\[
\Phi_{\cV}=\frac{1}{2}\gamma_{\cF^\circ}^2
\AND \Phi_\cV^*=\frac{1}{2}\gamma_\cF^2
\]
with $\cF=\set{LZ}{LL^T\in \cV\cap \bS_+^n,\; Z\in \bB_F}$.
\end{proposition}

\begin{proof}  
Lemma \ref{lem:ccq=sccq} tells us that assumption (i) is equivalent to SCCQ, and 
Corollary \ref{cor:Indicator PCQ} tells us that assumption (ii) is equivalent to BPCQ. 
Hence, by
Theorem \ref{th:indicatorConjugateGeneral}, either assumption (i) or (ii)
implies that $\Phi_{\cV}^* =p^{**}=p$.
The remainder is now follows from the definition of $p$, 
equation \eqref{eq:DefPhi}, and Corollary \ref{cor:Gauge case II}.
\end{proof}

 \noindent
Next consider the subdifferential of a VGF when defined by \eqref{eq:VGF}. Although, 
a VGF is always convex, we take the {\em convex-composite}  perspective, see e.g. \cite{BuP 92}, since a VGF is simply the composition of a closed, proper, convex function $\sigma_{\cV\cap \bS^n_+}$ and a nonlinear map $H:Y\mapsto YY^T$.  
The {\em basic constraint qualification} for 
the composition $\Phi_{\cV}=\frac{1}{2}\sigma_{\cV\cap \bS^n_+}\circ H$
at a point $\bar Y\in \dom \Phi_{\cV}$
is given by 
 \begin{equation*}
 \!\!\! \!\!\! \!\!\!\mbox{(BCQ)}\qquad\qquad
 N_{\dom \sigma_{\cV\cap \bS^n_+}}(\bar Y\bar Y^T)\cap  (\Ker_n \bar Y^T)=\{0\}.
\end{equation*}
It is well-known that this condition is essential for a full subdifferential calculus of convex-composite functions \cite{RoW98}.  We now show that this condition
is intimately linked to condition (ii) in Corollary \ref{cor:Indicator PCQ}.

\begin{lemma}[BPCQ$=$PCQ$=$BCQ for VGFs]\label{lem:BCQ VGF} Let $\Phi_{\cV}$ be as in  \eqref{eq:VGF} and assume that $\bS^n_+\cap \cV\neq \emptyset$. Then the following are equivalent:
\begin{itemize}
\item[(i)] There exists  $\bar Y\in \dom \Phi_{\cV}$ such that BCQ holds;
\item[(ii)] ((B)PCQ)  $\mathcal{V}^\infty \cap \bS^n_{+}= \{0\}$ (or equivalently $\cV\cap \bS^n_+$ is bounded);
\item[(iii)] BCQ holds at every $\bar Y\in \dom \Phi_{\cV}$.
\end{itemize} 
\end{lemma}
\begin{proof}  '(i)$\Rightarrow$(ii)':    
Let $\bar V\in \bS^n_+\cap \cV$ and assume (ii) is violated, i.e. there exists $0\neq W\in (\cV\cap \bS^n_+)^\infty=\cV^\infty\cap \bS^n_+$. 
By  
\eqref{eq:horizon recession}, we have 
\begin{equation}\label{eq:V_t}
V_t:=\bar V+tW \in \cV\cap \bS^n_+\quad (t>0).
\end{equation}
Now, take any $\bar Y\in \dom \Phi_{\cV}$. Then, for all $t>0$, we have
\begin{eqnarray*}
+\infty &  > & \Phi_{\cV}(\bar Y)\\
& = & \sup_{V\in \bS^n_+\cap \cV} \ip{V}{\bar Y\bar Y^T}\\
& \geq & \ip{V_t}{\bar Y\bar Y^T}\\
& \geq & t\ip{W}{\bar Y\bar Y^T}.
\end{eqnarray*}
Since $W\succeq 0$, we have $\ip{\bar Y\bar Y^T}{W}=\tr(\bar Y^TW\bar Y)\geq 0$. In view of the above chain of inequalities this implies $\ip{W}{\bar Y\bar Y^T}=0$ and as 
$W, \;\bar Y\bar Y^T\in \Sn_+$ this gives $W\bar Y\bar Y^T=0$. Since $ \rge \bar Y=\rge \bar Y\bar Y^T$ this implies $W\bar Y=0$ or, equivalently, $\bar Y^TW=0$. 
Therefore, we have 
$0\neq W\in (\cV\cap \bS^n_+)^\infty \cap (\Ker_n \bar Y^T)$. 
Now, observe  that $N_{\dom \sigma_{\cV\cap \bS^n_+}}(Z)=(\cV\cap \bS^n_+)^\infty$ for any $Z\in \dom \sigma_{\cV\cap \bS^n_+}$, see e.g. \cite{RoW98}. This shows that BCQ is violated at $\bar Y$. Since $\bar Y\in \dom \Phi_{\cV}$ was chosen arbitrarily, this establishes the desired implication.

\noindent 
'(ii)$\Rightarrow$(iii)':  If $\cV\cap \bS^n_+$ is bounded, then $\dom \sigma_{\cV\cap \bS^n_+}=\bS^n$, and so, for every $\bar Y \in \dom \Phi_{\cV}$,
$N_{\dom \sigma_{\cV\cap \bS^n_+}}(\bar Y\bar Y^T)=\{0\}$ giving the desired implication.
\smallskip

\noindent 
'(iii)$\Rightarrow$(i)': Obvious.  
\end{proof}

\noindent
We now derive the formula for the subdifferential of the VGF from \eqref{eq:VGF}. 

\begin{proposition}\label{prop:VGF subdiff} Let $\Phi_{\cV}$ be given by \eqref{eq:VGF}.  Then 
\[
\p \Phi_{\cV}(\bar Y)\supset \set{\bar V\bar Y }{\bar V\in \cV\cap \bS^n_+:\;\ip{\bar V}{\bar Y\bar Y^T}=\Phi_{\cV}(\bar Y)}\quad(\bar Y\in \dom \Phi_{\cV}).
\]
If  $\bS^n_+\cap\cV$ is nonempty and bounded,  equality holds and $\dom \Phi_{\cV}=\R^{n\times m}$.
\end{proposition}
\begin{proof} Combine  Lemma \ref{lem:BCQ VGF} with \cite[Theorem 10.6]{RoW98}, \cite[Corollary 8.25]{RoW98} and the fact that for $H:Y\to YY^T$ we have $\nabla H(Y)^*V=2VY$ for all $(Y,V)\in \bE$.
\end{proof}

\noindent
We next consider an example. 

\begin{example}[Failure of subdifferential calculus for VGF]\label{ex:Failure}  Let $\cV:= \pos \{I\}\subset \bS^n$, put $m:=1$ and let $H:Y\mapsto YY^T$.  Then clearly condition (i) in Proposition \ref{prop:VGF conjugate} holds, but condition (ii) and hence the BCQ fails. 
We have 
\begin{equation}\label{eq:Sig Ex}
\sigma_{\cV\cap \bS^n_+}(W)=\sup_{\alpha \geq 0} \alpha \tr(W)=\delta_{\set{U\in \bS^n}{\tr(U)\leq 0}}(W)\quad (W\in \bS^n).
\end{equation}
Hence, we obtain $\dom \Phi_{\cV}=\{0\}$ and 
$
\nabla H(0)^*\p \sigma_{\cV\cap \bS^n_+}(0)=\{0\}.
$
On the other hand, we have 
$
\Phi_{\cV}=\frac{1}{2}\sigma_{\cV\cap \bS^n_+}\circ H=\delta_{\{0\}}.
$
Therefore,
\[
\p \Phi_{\cV}(0)=N_{\{0\}}(0)=\R^{n\times m}\supsetneq \{0\}= \nabla H(0)^*\p \sigma_{\cV\cap \bS^n_+}(0).
\]

\end{example}

\noindent
Example \ref{ex:Failure} establishes  various things: First, it shows  that  condition (i) in Proposition \ref{prop:VGF conjugate} does not yield equality in the subdifferential formula for VGFs. It also illustrates that equality in the subdifferential formula may fail tremendously in the absence of BCQ, even for a convex-composite which is, in fact, convex. 

Jalali, Fazel and Xiao \cite{JFX17} employ great effort to 
compute  the conjugate of a (convex) VGF, cf. the proof of 
\cite[Proposition 7]{JFX17}. 
However, a slightly refined version of \cite[Proposition 7]{JFX17} follows 
immediately from our analysis.


\begin{proposition}[Subdifferential of $\Phi_{\cV}^*$]\label{prop:VGF subdifferential conjuagte} Let $\Phi_{\cV}$ be given by \eqref{eq:VGF}. 
\begin{itemize}
\item[(a)] ((S)CCQ) If $\cV\cap \bS^n_{++}\ne\emptyset$,
$\dom \p \Phi^*_\cV=\dom \Phi_{\cV}^*$ and 
\[
	\partial \Phi_{\cV}^*(\bar X)
	=\argmax_Y\left\{ \ip{\bar X}{Y}-\inf_{\half YY^T\preceq T}\sigma_{\cV\cap\Sn_+}{(T)}\right\}.
\]
%
%
\item[(b)] ((B)PCQ) If the set $\cV\cap\bS^n_{+}$ is nonempty and bounded,
$\dom \p \Phi^*_\cV=\dom \Phi_{\cV}^*$ and we have
\[
\partial \Phi_{\cV}^*(\bar X)=\set{\bar Y}{
\begin{aligned}&\exists \bar V\in\cV\cap \bS^n_+: \rge \bar X\subset \rge \bar V,\;  
\\
&\Phi_{\cV}^*(\bar X)=\frac{1}{2}\tr\left(\bar X^T\bar V^\dagger \bar X\right)=\ip{\bar X}{\bar Y}-\Phi_{\cV}(\bar Y),
\end{aligned}}
\]
for all $\bar X\in \dom \Phi_{\cV}^*$.
\end{itemize}
%
%
%
\end{proposition}

\begin{proof} 
(a) By Lemma \ref{lem:ccq=sccq}, PCQ$=$SCCQ and 
$\Phi_{\cV}=p^*$.  
The subdifferential formula follows from Proposition \ref{prop:SD + PS} (a) (see in particular the third identity in (c)). 
\smallskip

\noindent
(b)
The fact that $\dom \p \Phi^*_\cV=\dom \Phi_{\cV}^*$ is due to the fact that the latter is a subspace, hence relatively open, cf. Lemma \ref{lem:p} (c). The remainder follows from Lemma \ref{lem:ccq=sccq} and Proposition \ref{prop:SD + PS} (c).
\end{proof}

\subsection{VGFs and squared Ky Fan norms}\label{sec:VGF and Ky Fan}
For $p \ge 1$, $1 \le k \le \min\{m, n\}$, the \emph{Ky Fan (p,k)-norm} \cite[Ex. 3.4.3]{HJ91} of a matrix $X \in \mathbb R^{n\times m}$ is defined as
$$
\|X\|_{p,k} = \left( \sum_{i=1}^k \sigma_i^p \right)^{1/p},
$$
where $\sigma_i$ are the singular values of $X$ sorted in nonincreasing order. 
In particular, the $(p, \min\{m,n\})$-norm is the Schatten-p norm
and the $(1, k)$-norm is the standard Ky Fan k-norm, see \cite{HJ91}.
For $1\le p\le\infty$, denote the closed unit
ball for $\|\cdot\|_{p,k}$ by $\bB_{p,k}:=
\set{X}{\|X\|_{p,k}\le 1}$.
For $1\le p\le \infty$, define $s:=p/2$. Then, for $2\le p\le \infty$,
%
we have
\[\begin{aligned}
\half \|X\|^2_{p,k}&=
\half\left[ \sum_{i=1}^k(\sig_i^2)^s\right]^{1/s}
\\ &
= \half\|XX^T\|_{s,k}=\half\sig_{\bB^\circ_{s,k}}(XX^T)
=\half\sig_{\bB^\circ_{s,k}\cap \bS^n_+}(XX^T)
\\ &
=\half\Omega_{\bB^\circ_{s,k}}(X),
\end{aligned}
\]
where the first equality follows from the definition of $s$, the second
from the definition of the singular values, the third from properties 
of gauges and their polars, the fourth from 
the equivalence $\ip{V}{XX^T}=\sum_{j=1}^mx_j^TVx_j$ with 
the $x_j$'s the columns of $X$,
and the final from \eqref{eq:VGF}.
For the Schatten norms, where $k=\min\{n,m\}$ we have 
$\bB^\circ_{s,k}= \bB_{\hat s,k}$,
where $\hat s$ satisfies $\frac{1}{s}+\frac{1}{\hat s}=1$, see \cite{HJ85}. 
For other values of $k$, the representation
of $\bB^\circ_{s,k}$ can be significantly more complicated,
e.g. see \cite{DoV15}.

\section{Final remarks}\label{sec:Final}
\noindent
We studied  partial infimal projections of the generalized matrix-fractional function  with a closed, proper, convex function  $h:\bS^n\to\rbar$. Sufficient conditions for closedness and properness  as well  as representations of both the conjugate and the subdifferential  of the infimal projections under  the associated essential constraint qualifications.  
The general results were applied to the cases when 
$h$ is a support or an indicator function of a closed, convex set in $\bS^n$. 
These results revealed close connections to a range of important convex functions on $\Rnm$.
In particular, the infimal projection with linear functionals 
yielded smoothing variational representations for the family of scaled nuclear norms, while 
the infimal projection with an indicator is often a squared gauge.  
As a special case,  it was shown that the conjugate  of the infimal projection coincides with a variational Gram function (VGF) of the underlying set. 
Hence the variational calculus for VGFs follows easily as a consequence of our  general study. 
In all of these cases, the infimal projection opens the door to new smoothing approaches
to a range of nonsmooth optimization problems on $\Rnm$ using the representation
\eqref{eq:embedd optimization}.

\section{Appendix}\label{sec:Appendix} 
\noindent
In what follows we  use the  {\em direct sum} of  functions $f_i: \cE\to \eR\;(i=1,\dots,m)$  which is defined by 
\[
\oplus_{i=1}^m f_i:\cE^m\to \eR, \quad \oplus_{i=1}^m f_i(x_1,\dots,x_m)=\sum_{i=1}^m f_i(x_i).
\]

\begin{theorem}[Extended sum rule]\label{prop:ExtSumRule General}
Let $f_i\in \Gamma_0(\cE)\; (i=1,\dots,m)$ and set  $f:=\sum_{i=1}^mf_i$.
Then the following hold:
\begin{enumerate}
\item[(a)]  The conjugate of $f$ is given by  $f^*=\cl(f^*_1\infconv f^*_2\infconv\dots\infconv f^*_m)$.  Under the condition 
\begin{equation}\label{eq:Rel Int CQ General}
\bigcap_{i=1}^m\ri(\dom f_i)\ne \emptyset
\end{equation}
we have  $f^*=f^*_1\infconv f^*_2\infconv\dots\infconv f^*_m$ which is closed,  proper and convex  and
\[
\emptyset\ne \cT(z):=\argmin\set{\sum_{i=1}^mf_i^*(z^i)}{z=\sum_{i=1}^mz^i}
\quad ( z\in\dom f^*).
\]
\item[(b)]
If $\bz\in\sum_{i=1}^m\p f_i(\bx)$, then $\cT(\bz)\ne\emptyset$ and
\[
\cT(\bz)=\set{(z^1,\dots,z^m)}{\bz=\sum_{i=1}^mz^i,\ z^i\in\p f_i(\bx),\ i=1,\dots,m}.
\]
\item[(c)]  Under \eqref{eq:Rel Int CQ General} we have  $\p f = \sum_{i=1}^m\p f_i$,
$\dom\p f =\bigcap_{i=1}^m\dom \p f_i$ and 
\[
\begin{aligned}
\p f(\bx)&=\set{\sum_{i=1}^mz^i}{z^i\in\p f_i(\bx), i=1,\dots,m}
\\ &= \set{\bz}
{(z^1,\dots,z^m)\in\cT(\bz),\; z^i\in\p f_i(\bx)\ i=1,\dots m}\; (\bx\in \dom\p f). 
\end{aligned}
\]
\item[(d)]  Under \eqref{eq:Rel Int CQ General}, 
$f^*=f^*_1\infconv f^*_2\infconv\dots\infconv f^*_m$, 
$\dom \p f^*=\set{z}{\emptyset\ne \cT(z)}\ne\emptyset$, and
\[
\p f^*(\bz)=\set{\bigcap_{i=1}^m\p f_i^*(z^i)}{\bz=\sum_{i=1}^mz^i}\quad (\bz\in\dom\p f^*).
\]
\end{enumerate}
\end{theorem}
\begin{proof}
(a)  See \cite[Theorem 16.4]{RTR70}.
\smallskip

\noindent
(b) Let $L:\cE^m\to \cE$ be defined by    $L(z^1,\dots,z^m)= \sum_{i=1}^mz^i$. Then its adjoint $L^*:\cE\to\cE^m$ is given by  $L^*(x)=(x,\dots,x)\; (x\in \cE)$.
Let $\bz\in\sum_{i=1}^m\p f_i(\bx)$, and take any $z^i\in\p f_i(\bx)\; (i=1,\dots,m)$ such that
$\bz=\sum_{i=1}^mz^i$. By \cite[Theorem 23.5]{RTR70}, 
$\bx\in\p f^*_i(z^i)\; (i=1,\dots,m)$. Hence, by 
\cite[Theorem 23.8, 23.9]{RTR70} and \cite[Proposition 16.8]{BaC11} we obtain 
\[
0\in \rge L^*  +\p f^*_1(z^1)\times \cdots \times \p f^*_m(z^m)\subset
\p(\indicator{L(\cdot)-\bz}{\{0\}}+\oplus_{i=1}^m f^*_i)(z^1,\dots,z^m).
\]
Therefore, $(z^1,\dots,z^m)\in\cT(\bz)$, and we have
\[
\emptyset\ne \set{(z^1,\dots,z^m)}{\bz=\sum_{i=1}^mz^i,\ z^i\in\p f_i(\bx),\ i=1,\dots,m}
\subset \cT(\bz).
\]
To see the reverse inclusion, let $(z^1,\dots,z^m)\in\cT(\bz)$. By assumption and again  \cite[Theorem 23.8]{RTR70}, we have 
$\bz\in\sum_{i=1}^m\p f_i(\bx)\subset \p f(\bar x)$. 
By \cite[Theorem 23.5]{RTR70} and the fact that 
$f^*(\bar z)=\sum_{i=1}^mf_i^*(z^i)$, we have
\[
\sum_{i=1}^m\ip{z^i}{\bx}=\ip{\bz}{\bx}
= f^*(\bar z)+f(\bx)
 =\sum_{i=1}^m(f_i^*(z^i)+f_i(\bx)),
\]
so that
\(
0= \sum_{i=1}^m(f_i^*(z^i)+f_i(\bx)-\ip{z^i}{\bx}).
\)
By the 
Fenchel-Young inequality, $f_i^*(z^i)+f_i(\bx)-\ip{z^i}{\bx}\ge 0 \;(i=1,\dots, m)$, 
hence equality must hold for each $i=1,\dots,m$, or equivalently
$z^i\in\sd f_i(\bx)\;(i=1,\dots, m)$. This establishes the reverse inclusion.
\smallskip

\noindent
(c) The first two consequences follow from \cite[Theorem 23.8]{RTR70}.
For the third, the first equivalence simply follows from the fact that
$\p f=\sum_{i=1}^m\p f_i$. To see the second equivalence, let $\bz\in\p f(\bx)$.
Then, by part (b), $\cT(\bz)\ne\emptyset$, and, for every $(z^1,\dots,z^m)\in\cT(\bz)$,
we have $z^i\in\p f_i(\bx),\ i=1,\dots,m$. Hence,
\[\p f(\bx)\subset\set{\bz}{(z^1,\dots,z^m)\in\cT(\bz),\ z^i\in\p f_i(\bx),\ i=1,\dots,m}.\]
The reverse inclusion follows from the first equivalence.
\smallskip

\noindent
(d) By (a), $f^*=f^*_1\infconv f^*_2\infconv\dots\infconv f^*_m\in \Gamma_0(\cE)$ and $\cT(z)\ne\emptyset$ for all $z\in\dom f^*$.

Let us first suppose that $\bz\in\dom\p f^*\subset\dom f^*$, then 
$\cT(\bz)\ne\emptyset$. Let $\bx\in\p f^*(\bz)$. By \cite[Theorem 23.5]{RTR70},
$\bz\in\p f(\bx)$. By part (c), this is equivalent to the existence of $z^i\in\p f_i(\bx)$
such that $\bz=\sum_{i=1}^mz^i$, which, by \cite[Theorem 23.5]{RTR70},
is equivalent to $\bx\in \set{\bigcap_{i=1}^m\p f_i^*(z^i)}{\bz=\sum_{i=1}^mz^i}$.
Hence $\p f^*(\bz)\subset \set{\bigcap_{i=1}^m\p f_i^*(z^i)}{\bz=\sum_{i=1}^mz^i}$.

On the other hand, let $\bx\in \set{\bigcap_{i=1}^m\p f_i^*(z^i)}{\bz=\sum_{i=1}^mz^i}$.
Then, by \cite[Theorem 23.5]{RTR70}   we have  $\bz\in\p f(\bx)$. But then, again by \cite[Theorem 23.5]{RTR70}, $\bx\in\p f^*(\by)$.
Finally, suppose that $(z^1,\dots,z^m)\in\cT(\bz)\ne\emptyset$. 
Then, as in part (a), 
\(
0\in  \rge L^*  +\p f^*_1(z^1)\times \cdots \times \p f^*_m(z^m),
\)
or equivalently, there is an $\bx$ such that 
$\bx\in\bigcap_{i=1}^m\p f^*_1(z^i)$ with $\bz=\sum_{i=1}^mz^i$, i.e.,
$\bx\in\p f^*(\bz)$. This completes the proof.
\end{proof}

%
%
\begin{proposition}[Partial conjugates]\label{prop:Partial Conj} Let 
$f\in \Gamma(\cE_{1}\times \cE_2)$ and $\bar x\in\cE_1$ be such that 
$\bar g:=f(\bar x,\cdot)$ is proper and  $\bar x\in \ri L(\dom f)$, where $L:(x,v)\mapsto x$.
Then 
\[
 \bar g^*(w)=\inf_{z: (z,w)\in \dom f^*} [f^*(z,w)-\ip{\bar x}{z}].
\]
\end{proposition}
\begin{proof} 
%
%
By \cite[Theorem 6.6]{RTR70}, $\ri L(\dom f)=L(\ri\dom f)$, so the hypothesis
implies the existence of a $\bar w\in\cE_2$ such that $(\bar x, \bar w)\in \ri\dom f$.
By \cite[Theorem 16.4]{RTR70},
\begin{eqnarray*}
\bar g^*(w)& = & \sup_{v} \{\ip{v}{w}- f(\bar x,w)\}\\
& = & \sup_{(x,v)}\{\ip{(x,v)}{(0,w)} -(f+\delta_{\{\bar x\}\times \cE_2})(x,v) \}\\
& =& (f+\delta_{\{\bar x\}\times \cE_2})^*(0,w)\\
& = & \cl(f^*\infconv \sigma_{\{\bar x\}\times \cE_2})(0,w),
\end{eqnarray*}
where the closure can be dropped if 
$\ri(\dom f)\cap\ri(\dom \delta_{\{\bar x\}\times \cE_2})\ne\emptyset$.
But this intersection is nonempty by hypothesis since
$(\bar x,\bar w)\in \{\bar x\}\times \cE_2=\ri(\dom \delta_{\{\bar x\}\times \cE_2})$. 
Hence 
\begin{eqnarray*}
\bar g^*(w)& = &(f^*\infconv \sigma_{\{\bar x\}\times \cE_2})(0,w) \\
& = & \inf_{(z,u)} \{f^*(z,u)+\ip{\bar x}{0-z}+\delta_{\{0\}}(w-u)\}\\
& = &  \inf_{z:(z,w)\in \dom f^*} \{f^*(z,w)-\ip{\bar x}{z}\}.
\end{eqnarray*}
\end{proof}

\bibliographystyle{amsplain}
\bibliography{references}

\end{document}